\documentclass[11pt, reqno]{amsart}    

\usepackage{amsmath,amssymb,latexsym}
\usepackage{color}
\usepackage{xcolor}
\usepackage{graphicx}
\usepackage{cite}
\usepackage[colorlinks=true]{hyperref}
\hypersetup{urlcolor=blue, citecolor=red, linkcolor=blue}

\newcommand{\R}{\mathbb{R}}

\newcommand{\Z}{\mathbb{Z}}

\newcommand{\bigO}{\mathcal{O}}
\newcommand{\diff}{\mathrm{d}}

\newtheorem{theorem}{Theorem}[section]
\newtheorem{lemma}[theorem]{Lemma}
\newtheorem{proposition}[theorem]{Proposition}

\newtheorem*{main-theorem}{Main Theorem}
\newtheorem*{remark*}{Remark}
\newtheorem*{lemma*}{Lemma A.1}

\numberwithin{equation}{section}

\begin{document}

\title[Generalized fifth-order KdV equation]{Global dynamics of the generalized fifth-order KdV equation with critical nonlinearity}

\author{Yuexun Wang}

\address{Universit\' e Paris-Saclay, CNRS, Laboratoire de Math\'  ematiques d'Orsay, 91405 Orsay, France.}

\email{yuexun.wang@universite-paris-saclay.fr}


\begin{abstract} We prove global existence and modified scattering for the solutions of the generalized fifth-order KdV equation with
critical nonlinearity for small and localized initial data. The proof is undergoing by using the space-time resonance method and the stationary phase argument.  
\end{abstract}
\maketitle

\section{Introduction}
We consider the following generalized fifth-order KdV equation with 
critical nonlinearity:
\begin{align}\label{eq:main}
u_t=\partial_x^5u+\alpha u^4u_x,
\end{align}
where \(u\) is a real function which maps \(\mathbb{R}\times \mathbb{R}\) to \(\mathbb{R}\), and \(\alpha=1\) (defocusing case) or \(\alpha=-1\) (focusing case). The equation \eqref{eq:main} is a member of the general fifth-order KdV equations 
\begin{align*}
u_t=\alpha_1\partial_x^5u+\alpha_2\partial_x^3u+\partial_xg(u,\partial_xu,\partial_x^2u),\quad \alpha_1,\alpha_2\in\R,  \alpha_1\neq 0,
\end{align*}
which model plasma waves and capillary-gravity waves (see the introduction of \cite{MR1946769,MR3702717} for a useful survey). On the other hand, the equation \eqref{eq:main} is
also interesting mathematically due to its critical dispersive nature in the sense that the asymptotics of its large time solutions differ from the linear solutions of its linear equation, which requires a nonlinear phase correction (modified scattering) in studying large time behavior of solutions of the equation \eqref{eq:main}. One can consult the pioneering work \cite{Ozawa} for the concept of modified scattering.   

It is standard to show that the Cauchy problem of \eqref{eq:main} is well posed in \(C\big([-T,T];H^s(\R)\big)\) (\(s>\frac{3}{2}\)) for a short time \(T>0\), for instance, one may refer to \cite{MR1044731,MR533234}. Our aim in the present work is to study the global existence and modified scattering for the solutions of \eqref{eq:main} with small and localized initial data, in the frame work of the space-time resonance method \cite{MR2482120,MR2850346} and the stationary phase argument \cite{MR2850346}. The main ingredient is to study the evolutionary equation of the profile of the solutions in Fourier space which is unfolded by a careful stationary phase analysis based on an adaptation of the argument of \cite{MR3519470}.  
Since we only focus on small solutions, the sign of \(\alpha\) will not matter, and will be taken to be '\(-1\)' in the rest of the paper. 
Our main result can be stated precisely as follows:

\begin{theorem}\label{th:main} Given the initial data \(u_0\) as
\begin{align}\label{eq:initial}
u(x,0)=u_0(x).
\end{align}	
Assume that \(u_0\) satisfies 
\begin{align}\label{1}
\|u_0\|_{H^2(\R)}+\|xu_0\|_{L^2(\R)}\leq \varepsilon_0\leq \overline{\varepsilon},
\end{align}
for some constant \(\overline{\varepsilon}\) sufficiently small. 
Then the Cauchy problem \eqref{eq:main}-\eqref{eq:initial} admits a unique global solution
\(u\in C\big(\R; H^2(\R)\big)\) satisfying the decay estimates for \(t\geq 1\) and \(x\in\R\)
	\begin{equation}\label{2}
	\begin{aligned}
	\left||\partial_x|^\beta u(x,t)\right|\lesssim \varepsilon_0t^{-(\beta+1)/5}\langle x/t^{1/5}\rangle^{-\frac{3}{8}+\frac{\beta}{4}},\quad \beta\in[0,3].
	\end{aligned}
	\end{equation}
Moreover,  the solution has the following asymptotics as \(t\rightarrow+\infty\):

	\((\text{Decaying region})\) When \(x\geq t^{1/5}\), we have the decay estimate
	\begin{equation}\label{3}
	\begin{aligned}
	|u(x,t)|\lesssim \varepsilon_0 t^{-1/5}(x/t^{1/5})^{-7/8}.
	\end{aligned}
	\end{equation}

	\((\text{Self-similar region})\) When \(|x|\leq t^{\frac{1}{5}+4\gamma}\), with \(\gamma=\frac{1}{5}\big(\frac{1}{10}-C\varepsilon_0^\frac{2}{5}\big)\), the solution is approximately self-similar
	:
	\begin{equation}\label{4}
	\begin{aligned}
	|u(x,t)-t^{-1/5}Q(x/t^{1/5})|\lesssim \varepsilon_0 t^{-\frac{1}{5}-\frac{7\gamma}{2}},
	\end{aligned}
	\end{equation}
	where \(Q\) is a bounded solution of the nonlinear ordinary differential equation
	\begin{equation}\label{5}
	\begin{aligned}
	Q^{(4)}-5^{-1} xQ-Q^5=0
	\end{aligned}
	\end{equation}
with 	
	\begin{equation}\label{6}
	\begin{aligned}
	\|Q\|_{L^\infty(\R)}\lesssim \varepsilon_0.
	\end{aligned}
	\end{equation}

	\((\text{Oscillatory region})\) When \(x\leq -t^{\frac{1}{5}+4\gamma}\), the solution has a nonlinearly modified asymptotic behavior: there exists \(f_\infty\in L^\infty(\R)\) such that 
	\begin{equation}\label{7}
	\begin{aligned}
	&\left|u(x,t)-\frac{1}{\sqrt{5t\xi_0^3}}\Re\left\{ \exp\left(-4\mathrm{i}t\xi_0^5+\frac{\mathrm{i}\pi}{4}+\frac{\mathrm{i}}{40t\xi^5}|f_\infty(\xi_0)|^4\right)f_\infty(\xi_0)\right\}\right|\\
	&\lesssim \varepsilon_0 t^{-1/5}(-x/t^{1/5})^{-9/20},
	\end{aligned}
	\end{equation}
	where \(\xi_0:=\sqrt[4]{-x/(5t)}\), and \(\Re\) denotes the real part. 
	
\end{theorem}

Since the equation \eqref{eq:main} is time-reversible, the asymptotics for \(t\rightarrow-\infty\) follows immediately. We mention that the proof presented in this work doesn't rely on the complete integrability, thus may be applied to a wider class with short range perturbations of the nonlinearity. 

Throughout the paper, we will always use \(f(t)=e^{-t\partial_x^5}u(t)\) to denote the profile of \(u\). By time reversibility we solely need to consider the existence for positive time.
The local well-posedness on the time interval \([0,1]\) for \eqref{eq:main}-\eqref{eq:initial} is standard provided \(\|u_0\|_{H^s}\) (\(s>\frac{3}{2}\)) is sufficiently small, in particular under the smallness assumption \eqref{1}. Then the existence and uniqueness of global solutions may be constructed by a bootstrap argument which guarantees us to extend the local solutions. More precisely,
assume that the following \(X\)-norm is a priori small:
\begin{equation}\label{8}
\begin{aligned}
\|u\|_{X}&=\sup_{t\geq 1}\bigg(\|u(t)\|_{H^2}+t^{-1/10}\|xf(t)\|_{L^2}+\|\widehat{f}(t,\xi)\|_{L_\xi^\infty}\bigg)\leq \varepsilon_1
\end{aligned}
\end{equation}
with \(\varepsilon_1=\varepsilon_0^{1/5}\), 
we then aim to show the above a priori assumption may be improved to
\begin{align}\label{9}
\|u\|_{X}\leq C(\varepsilon_0+\varepsilon_1^5),
\end{align}
for some absolute constant \(C>1\). Moreover we may choose \(\overline{\varepsilon}:=(4C)^{-5/4}\) as an upper bound of \(\varepsilon_0\) in Theorem \ref{th:main}.

There are other approaches dealing with asymptotics for large time solutions of dispersive PDEs: via inverse scattering transform \cite{MR1207209} for large solutions but relying on the complete integrability of the equation, and using PDE techniques \cite{MR1687327,MR3561525,MR3462131,MR3382579} not relying on the complete integrability of the equation but restricting to small solutions, one can refer to  \cite{MR3445510} for a comparative survey on related results. 

We learned that the large time behavior of solutions of \eqref{eq:main} was also carefully studied recently \cite{MR3813999} in which the author got a precise asymptotic behavior of the solutions following the testing by wave packets argument \cite{MR3462131,MR3382579}.  For other studies on the fifth-order KdV equations, one may refer to \cite{MR3301874} and references therein.
Nevertheless, we prefer to present an alternative approach to understand the asymptotics of the large time solutions of \eqref{eq:main} since the argument we are using is also flexible in fractional dispersive models (e.g. \cite{SW}). Different with \cite{MR3813999}, our proof is fully carried out in Fourier space whose main idea is to identify the ODE of \(\partial_t\widehat{f}(t,\xi)\) precisely. For this, inspired by \cite{MR3519470}, we decompose the critical nonlinearity into the stationary phase part and the non-stationary phase part in a crucial way, in which the stationary phase part contains the leading term of \(\partial_t\widehat{f}(t,\xi)\) that determines the modified asymptotic behavior of the solutions. 
More precisely, we will derive  
\begin{equation*}
\begin{aligned}
\partial_t\widehat{f}(t,\xi)&=\bigg(\frac{-\mathrm{i}}{40t^2\xi^5}|\widehat{f}(t,\xi)|^4\widehat{f}(t,\xi)+\frac{c_1\mathrm{i}}{t^2\xi^5}e^{-\frac{624\mathrm{i}t\xi^5}{625}}\widehat{f}(t,\xi/5)^5\\
&\quad+\frac{c_2\mathrm{i}}{t^2\xi^5}e^{-\frac{80\mathrm{i}t\xi^5}{81}}|\widehat{f}(t,\xi/3)|^2\widehat{f}(t,\xi/3)^3\bigg){\bf{1}}_{|\xi|>t^{-1/5}}\\
&\quad+\{integrable\ terms\}.
\end{aligned}
\end{equation*}
Since the second and third terms are not time resonant, one can apply integration by parts in time to handle them. However the first term is not integrable, one needs to remove it from the ODE by using an integrating factor which leads to the phase correction in the asymptotic behavior of the solutions.

Comparison to the modified KdV equation \cite{MR3519470},
the space-time resonance analysis in our case is more complicated due to the structure of the phase function in the Duhamel's formula. 
We also mention that one can further study the stability of soliton solutions under small perturbations of \eqref{eq:main} after Theorem \ref{th:main} (the stability of the zero solution under small perturbations) as \cite{MR3519470}, however which is an independent interest and will be unfolded in a different work.

In Section \ref{sec:1}, we derive some crucial estimates on the semigroup generated by the linear part of the equation \eqref{eq:main}. In Section \ref{sec:2}, we estimate the Sobolev norm and wighted Sobolev norm in \eqref{8}. 
Section \ref{sec:3} is devoted to controlling \(\widehat{f}\) in \(L^\infty\)-norm in \eqref{8}. We study asymptotic behavior in Section \ref{sec:4}.   

  \bigskip
  
\noindent{\bf{Notations.}}
We finally list some notations frequently used throughout the paper. Let \(L^p(\R)\) (\(p \in [1,\infty]\)) be the standard Lebesgue spaces, in particular, \(L^2(\R)\) is a Hilbert space with inner product 
\[
(g,h)_2:=\int_{\R}gh \, \diff x.
\]
Similarly, let $H^s(\R)$ (\(s >0\)) be the usual Sobolev spaces with norm 
\[
\|g\|_{H^s(\R)}: = \|(1-\partial_x^2)^{s/2} g \|_{L^2(\R)},
\]
and  let $C\big([0,T]; H^s(\R)\big)$ be the space of all bounded continuous functions $g\colon [0,T]\rightarrow H^s(\R)$ normed by
\[
\| g \|_{C\big([0,T]; H^s(\R)\big)}: =  \sup_{t \in [0,T]} \|g(t,\cdot)\|_{H^s(\R)}. 
\]
We denote by \(\mathcal{F}(g)\) or \(\widehat{g}\) the Fourier transform of a Schwartz function \(g\) whose formula is given by 
\begin{equation*}
\begin{aligned}
\mathcal{F}(g)(\xi)=\widehat{g}(\xi):=\frac{1}{\sqrt{2\pi}}\int_{\R}g(x)e^{-\mathrm{i}x\xi}\,\diff x
\end{aligned}
\end{equation*}
with inverse
\begin{equation*}
\begin{aligned}
\mathcal{F}^{-1}(g)(x)=\frac{1}{\sqrt{2\pi}}\int_{\R}g(\xi)e^{\mathrm{i}x\xi}\,\diff \xi,
\end{aligned}
\end{equation*}
and by \(m(\partial_x)\) the Fourier multiplier with symbol \(m\) via the relation 
\begin{equation*}
\begin{aligned}
\mathcal{F}\big(m(\partial_x)g\big)(\xi)=m(\mathrm{i}\xi)\widehat{g}(\xi).
\end{aligned}
\end{equation*}

Take \(\varphi\in C_0^\infty(\R)\) satisfying \(\varphi(\xi)=1\) for \(|\xi|\leq 1\) and \(\varphi(\xi)=0\) when \(|\xi|>2\), and
let 
\begin{equation*}
\begin{aligned}
\psi(\xi)=\varphi(\xi)-\varphi(2\xi),\quad \psi_j(\xi)=\psi(2^{-j}\xi),\quad \varphi_j(\xi)=\varphi(2^{-j}\xi),
\end{aligned}
\end{equation*}
we then may define the  
Littlewood-Paley projections \(P_j,P_{\leq j},P_{> j}\)  via 
\begin{equation*}
\begin{aligned}
\widehat{P_jg}(\xi)=\psi_j(\xi)\widehat{g}(\xi),\quad \widehat{P_{\leq j}g}(\xi)=\varphi_j(\xi)\widehat{g}(\xi),\quad P_{> j}=1-P_{\leq j},
\end{aligned}
\end{equation*}
and also \(P_{\sim},P_{\lesssim j},P_{\ll j}\) by 
\begin{equation*}
\begin{aligned}
P_{\sim j}=\sum_{2^k\sim 2^j}P_k, \quad P_{\lesssim j}=\sum_{2^k\leq 2^{j+C}}P_k,\quad P_{\ll j}=\sum_{2^k\ll 2^j}P_k.
\end{aligned}
\end{equation*}
We will also denote \(g_j=P_jg, g_{\lesssim j}=P_{\lesssim j}g\), and so on, for convenience.

The notation \(C\)  always denotes a nonnegative universal constant which may be different from line to line but is
independent of the parameters involved. Otherwise, we will specify it by  the notation \(C(a,b,\dots)\).
We write \(f\lesssim g\) (\(f\gtrsim g\)) when \(f\leq  Cg\) (\(f\geq  Cg\)), and \(f \sim g\) when \(f \lesssim g \lesssim f\).
We also write \(\sqrt{1+x^2}=\langle x\rangle\) for \(x\in\R\) for simplicity.

\section{Linear estimates}\label{sec:1}
In this section, we aim to derive some crucial estimates on the semigroup generated by the linear part of the equation \eqref{eq:main}. These estimates will be applied later in closing energy estimates and proving asymptotics. Similar argument was used in \cite{MR3519470}.

\begin{lemma}\label{le:1}  Let \(t\geq 1\), \(x\in\R\) and \(g\) be a real function.  Then 
	\begin{equation}\label{10}
	\begin{aligned}
	\left|e^{t\partial_x^5}|\partial_x|^\beta g(x,t)\right|&\lesssim t^{-\beta/5}\langle x/t^{1/5}\rangle^{-\frac{3}{8}+\frac{\beta}{4}}\\
	&\quad\times\left(t^{-1/5}\|\widehat{g}\|_{L^\infty}
	+t^{-3/10}\|xg\|_{L^2}\right), \quad \text{for}\ \beta\in[0,3]. 
	\end{aligned}
	\end{equation}
Moreover,  we have the refined estimates:
when \(x\geq t^{1/5}\),  
\begin{equation}\label{11}
\begin{aligned}
\left|e^{t\partial_x^5}g(x,t)\right|\lesssim (x/t^{1/5})^{-7/8}\left(t^{-1/5}\|\widehat{g}\|_{L^\infty}
+t^{-3/10}\|xg\|_{L^2}\right),
\end{aligned}
\end{equation}
and for \(x\leq -t^{1/5}\), 
\begin{equation}\label{12}
\begin{aligned}
&\left|e^{t\partial_x^5}g(x,t)-\frac{1}{\sqrt{5t\xi_0^3}}\Re \exp\left(-4\mathrm{i}t\xi_0^5+\frac{\mathrm{i}\pi}{4}\right)\widehat{g}(t,\xi_0)\right|\\
&\lesssim (-x/t^{1/5})^{-9/20}\left(t^{-1/5}\|\widehat{g}\|_{L^\infty}
+t^{-3/10}\|xg\|_{L^2}\right),
\end{aligned}
\end{equation}
where \(\xi_0:=\sqrt[4]{-x/(5t)}\), and \(\Re\) denotes the real part. 
\end{lemma}

\begin{proof}[Proof of \eqref{10}] 
	We write
\begin{equation}\label{13}
	\begin{aligned}
	e^{t\partial_x^5}|\partial_x|^\beta g(x,t)&=\sqrt{\frac{2}{\pi}}\Re\int_0^\infty e^{\mathrm{i}t\Phi(\xi)}\xi^\beta\widehat{g}(t,\xi)\,\diff \xi,\\
	\Phi(\xi)&=\Phi(\xi;x,t):=t^{-1}x\xi+\xi^5.
	\end{aligned}
\end{equation}
Considering \(\xi\geq 0\) in \eqref{13}, so the phase function \(\Phi\) has only one stationary point \(\xi_0=\sqrt[4]{-x/(5t)}\) for \(x\leq0\).
The estimate \eqref{10} is much easier as \(x>0\) due to the absence of stationary points, and the following argument also applies. So	for \eqref{10},  it suffices to show 
	\begin{align}\label{14}
	\left|\int_0^\infty e^{\mathrm{i}t\Phi(\xi)}\xi^\beta\widehat{g}(t,\xi)\,\diff \xi\right|\lesssim t^{-\frac{1}{5}-\frac{\beta}{5}}\max(1,\xi_0t^{1/5})^{-\frac{3}{2}+\beta},
	\end{align}
for any \(t\geq1, x\leq0\) and any function \(g\) satisfying
		\begin{align}\label{15}
		\|\widehat{g}\|_{L^\infty}
		+t^{-1/10}\|xg\|_{L^2}\leq 1.
		\end{align}
There are two cases to  consider depending on the size of \(\xi_0\). 
	
	\noindent{\emph{Case 1: \(\xi_0\leq t^{-1/5}\)}.} It reduces to obtain a bound of \(t^{-\frac{1}{5}-\frac{\beta}{5}}\) for RHS of \eqref{14}. For this we split the integral in \eqref{14} by small and large frequencies:
	\begin{equation*}
	\begin{aligned}
	\int_0^\infty e^{\mathrm{i}t\Phi(\xi)}\xi^\beta\widehat{g}(t,\xi)\,\diff \xi
	&=\int_0^\infty e^{\mathrm{i}t\Phi(\xi)}\xi^\beta\widehat{g}(t,\xi)\varphi(2^{-10}t^{1/5}\xi)\,\diff \xi\\
	&\quad+\int_0^\infty e^{\mathrm{i}t\Phi(\xi)}\xi^\beta\widehat{g}(t,\xi)\big(1-\varphi(2^{-10}t^{1/5}\xi)\big)\,\diff \xi\\
	&=\colon A_1+A_2.
	\end{aligned}
	\end{equation*}
Using only the bound \(\|\widehat{g}\|_{L^\infty}\leq 1\) in \eqref{15}, one immediately gets 
	\begin{equation*}
	\begin{aligned}
	|A_1|\lesssim \|\widehat{g}\|_{L^\infty}\int_0^\infty \xi^\beta\varphi(2^{-10}t^{1/5}\xi)\,\diff \xi
	\lesssim t^{-\frac{1}{5}-\frac{\beta}{5}}.
	\end{aligned}
	\end{equation*}
For the second term \(A_2\),  we use an integration by parts to deduce
	\begin{equation*}
	\begin{aligned}
	&|A_2|\lesssim |A_{21}|+|A_{22}|,\\
	&A_{21}=t^{-1}\int_0^\infty \left|\partial_\xi\big[(\partial_\xi\Phi)^{-1}\xi^\beta\big(1-\varphi(2^{-10}t^{1/5}\xi)\big)\big]\widehat{g}(t,\xi)\right|\,\diff \xi,\\
	&A_{22}=t^{-1}\int_0^\infty \left|(\partial_\xi\Phi)^{-1}\xi^\beta\big(1-\varphi(2^{-10}t^{1/5}\xi)\big)\partial_\xi\widehat{g}(t,\xi)\right|\,\diff \xi.
	\end{aligned}
	\end{equation*}
Observing that \(|\partial_\xi\Phi|\gtrsim \xi^4\gtrsim t^{-4/5}\) on the support of the integral, we may estimate the resulting terms, by applying the first bound in \eqref{15} again to get 
	\begin{equation}\label{15.1}
	\begin{aligned}
	A_{21}&\lesssim t^{-1}\|\widehat{g}\|_{L^\infty}\int_0^\infty \bigg(\xi^{\beta-5}|1-\varphi(2^{-10}t^{1/5}\xi)|\\
	&\quad\quad\quad\quad\quad\quad+\xi^{\beta-4}|\varphi^\prime(2^{-10}t^{1/5}\xi)t^{1/5}|\bigg)\,\diff \xi\\
	&\lesssim t^{-1}t^{-\frac{1}{5}(\beta-4)}
	\lesssim t^{-\frac{1}{5}-\frac{\beta}{5}},
	\end{aligned}
	\end{equation}
and the second bound in \eqref{15} to obtain
	\begin{equation}\label{15.2}
	\begin{aligned}
	A_{22}&\lesssim t^{-1}\|\partial\widehat{g}\|_{L^2}\left(\int_0^\infty \left(\xi^{\beta-4}\big(1-\varphi(2^{-10}t^{1/5}\xi)\big)\right)^2\,\diff \xi\right)^{1/2}\\
	&\lesssim t^{-1}t^{1/10}\big(t^{-\frac{1}{5}(2\beta-7)}\big)^{1/2}
	\lesssim t^{-\frac{1}{5}-\frac{\beta}{5}}.
	\end{aligned}
	\end{equation}
	
\noindent{\emph{Case 2: \(\xi_0\geq t^{-1/5}\)}.} We need to show a bound of \(t^{-1/2}\xi_0^{-\frac{3}{2}+\beta}\) for RHS of \eqref{14} instead. 
Since the resonant contributions concentrate on \(\xi\sim \xi_0\), we split the integral in \eqref{14} as follows:
	\begin{equation*}
	\begin{aligned}
	\int_0^\infty e^{\mathrm{i}t\Phi(\xi)}\xi^\beta\widehat{g}(t,\xi)\,\diff \xi
	&=\int_0^\infty e^{\mathrm{i}t\Phi(\xi)}\xi^\beta\widehat{g}(t,\xi)\big(1-\psi(\xi/\xi_0)\big)\,\diff \xi\\
	&\quad+\int_0^\infty e^{\mathrm{i}t\Phi(\xi)}\xi^\beta\widehat{g}(t,\xi)\psi(\xi/\xi_0)\,\diff \xi\\
	&=\colon A_3+A_4.
	\end{aligned}
	\end{equation*}
We first control the non-stationary contributions. 
By integration by parts,  we bound \(A_3\) by
	\begin{equation*}
	\begin{aligned}
	&|A_3|\lesssim |A_{31}|+|A_{32}|,\\
	&A_{31}=t^{-1}\int_0^\infty \left|\partial_\xi\big[(\partial_\xi\Phi)^{-1}\xi^\beta\big(1-\psi(\xi/\xi_0)\big)\big]\widehat{g}(t,\xi)\right|\,\diff \xi,\\
	&A_{32}=t^{-1}\int_0^\infty \left|(\partial_\xi\Phi)^{-1}\xi^\beta\big(1-\psi(\xi/\xi_0)\big)\partial_\xi\widehat{g}(t,\xi)\right|\,\diff \xi.
	\end{aligned}
	\end{equation*}
Using the fact that \(|\partial_\xi\Phi|\gtrsim \max(\xi^4,\xi_0^4)\) on the support of the integral, and \eqref{15}, we estimate respectively
	\begin{equation}\label{15.3}
	\begin{aligned}
	|A_{31}|&\lesssim t^{-1}\|\widehat{g}\|_{L^\infty}\int_0^\infty \bigg(\max(\xi^{4},\xi_0^{4})^{-1}\xi^{\beta-1}|1-\psi(\xi/\xi_0)|\\
	&\quad\quad\quad\quad\quad\quad\quad\quad+\max(\xi^{4},\xi_0^{4})^{-1}\xi^\beta|\psi^\prime(\xi/\xi_0)|\xi_0^{-1}\bigg)\,\diff \xi\\
	&\lesssim t^{-1}\xi_0^{\beta-4},
	\end{aligned}
	\end{equation}
and
	\begin{equation}\label{15.4}
	\begin{aligned}
	|A_{32}|&\lesssim t^{-1}\|\partial\widehat{g}\|_{L^2}\left(\int_0^\infty \left(\max(\xi^{4},\xi_0^{4})^{-1}\xi^\beta\big(1-\psi(\xi/\xi_0)\big)\right)^2\,\diff \xi\right)^{1/2}\\
	&\lesssim t^{-9/10}\xi_0^{\beta-\frac{7}{2}}.
	\end{aligned}
	\end{equation}
Both bounds \(t^{-1}\xi_0^{\beta-4}\) and \(t^{-9/10}\xi_0^{\beta-\frac{7}{2}}\) are  better than the desired bound
\(t^{-1/2}\xi_0^{-\frac{3}{2}+\beta}\) due to \(\xi_0\geq t^{-1/5}\).

	It remains to study the stationary contributions.  Let \(l_0\) be the smallest integer with the property that \(2^{l_0}\geq t^{-1/2}\xi_0^{-3/2}\). Then the term \(A_4\) is dominated by
	\begin{equation*}
	\begin{aligned}
	|A_4|\leq \sum_{l=l_0}^{\log \xi_0+10}|A_{4l}|,
	\end{aligned}
	\end{equation*}
where
	\begin{equation*}
	\begin{aligned}
	&A_{4l_0}=\int_0^\infty e^{\mathrm{i}t\Phi(\xi)}\xi^\beta\widehat{g}(t,\xi)\psi(\xi/\xi_0)\varphi\big(2^{-l_0}(\xi-\xi_0)\big)\,\diff \xi,\\
	&A_{4l}=\int_0^\infty e^{\mathrm{i}t\Phi(\xi)}\xi^\beta\widehat{g}(t,\xi)\psi(\xi/\xi_0)\psi\big(2^{-l}(\xi-\xi_0)\big)\,\diff \xi,\quad l\geq l_0+1.
	\end{aligned}
	\end{equation*}
The desired bound for \(A_{4l_0}\) is immediate from the definition of \(l_0\), and the estimate
	\begin{equation*}
	\begin{aligned}
	|A_{4l_0}|\lesssim \|\widehat{g}\|_{L^\infty}\int_0^\infty \xi^\beta\psi(\xi/\xi_0)\varphi\big(2^{-l_0}(\xi-\xi_0)\big)\,\diff \xi
	\lesssim \xi_0^{\beta}2^{l_0}.
	\end{aligned}
	\end{equation*}
We are left finally to handle the terms \(A_{4l}\) for \(l\geq l_0+1\). Integration by parts yields
	\begin{equation*}
	\begin{aligned}
	&|A_{4l}|\lesssim |A_{4l,1}|+|A_{4l,2}|,\\
	&A_{4l,1}=t^{-1}\int_0^\infty \left|\partial_\xi\big[(\partial_\xi\Phi)^{-1}\xi^\beta\psi(\xi/\xi_0)\psi\big(2^{-l}(\xi-\xi_0)\big)\big]\widehat{g}(t,\xi)\right|\,\diff \xi,\\
	&A_{4l,2}=t^{-1}\int_0^\infty \left|(\partial_\xi\Phi)^{-1}\xi^\beta\psi(\xi/\xi_0)\psi\big(2^{-l}(\xi-\xi_0)\big)\partial_\xi\widehat{g}(t,\xi)\right|\,\diff \xi.
	\end{aligned}
	\end{equation*}
We first observe that \(|\partial_\xi\Phi|\sim \xi_0^32^l\) on the support of the integrals, and then estimate the integrations as \eqref{15.1}-\eqref{15.4} to obtain
	\begin{equation}\label{16}
	\begin{aligned}
	|A_{4l,1}|
	\lesssim t^{-1}\xi_0^{\beta-3}2^{-l},
	\end{aligned}
	\end{equation}
and
	\begin{equation}\label{17}
	\begin{aligned}
	|A_{4l,2}|
	\lesssim  t^{-9/10}\xi_0^{\beta-3}2^{-l/2}.
	\end{aligned}
	\end{equation}
Summing \eqref{16} over \(l\) immediately gives the desired bound. One obtains a bound of \(t^{-9/10}\xi_0^{\beta-3}2^{-l_0/2}\) by summing \eqref{17} in \(l\)  which is better than what we need. 	

\end{proof}

\begin{proof}[Proof of \eqref{11}] With the same assumption \eqref{15} here, recalling \eqref{13},  it is enough to show
	\begin{align*}
	\left|\int_0^\infty e^{\mathrm{i}t\Phi(\xi)}\widehat{g}(t,\xi)\,\diff \xi\right|\lesssim t^{-1/5}(x/t^{1/5})^{-7/8}.
	\end{align*}
	Using the fact that
	\[\partial_\xi\Phi=x/t+5\xi^4>0,\quad\text{for}\ x>0,\]
and integration by parts, we deduce
	\begin{equation*}
	\begin{aligned}
	\left|\int_0^\infty e^{\mathrm{i}t\Phi(\xi)}\widehat{g}(t,\xi)\,\diff \xi\right|&\lesssim t^{-1}\int_0^\infty \left|(\partial_\xi\Phi)^{-1}\partial_\xi\widehat{g}(t,\xi)\right|\,\diff\xi\\
	&\quad+t^{-1}\int_0^\infty \left|(\partial_\xi\Phi)^{-2}\partial_\xi^2\Phi\widehat{g}(t,\xi)\right|\,\diff\xi\\
	&=\colon B_1+B_2.
	\end{aligned}
	\end{equation*}
To bound the first term,  we use \eqref{15} to dominate
	\begin{equation*}
	\begin{aligned}
	|B_1|\lesssim t^{-1}\|\partial\widehat{g}\|_{L^2}\left(\int_0^\infty (x/t+5\xi^4)^{-2}\,\diff\xi\right)^{1/2}
	\lesssim t^{-3/10}(x/t^{1/5})^{-7/8}.
	\end{aligned}
	\end{equation*}
Similarly, we can estimate 
	\begin{equation*}
	\begin{aligned}
	|B_2|\lesssim t^{-1}\|\widehat{g}\|_{L^\infty}\int_0^\infty (x/t+5\xi^4)^{-2}\xi^3\,\diff\xi
	\lesssim t^{-1}(x/t)^{-1}.
	\end{aligned}
	\end{equation*}
This bound is better than the desired bound due to \(x\geq t^{1/5}\).

\end{proof}

	\begin{proof}[Proof of \eqref{12}] We will make the same assumption \eqref{15} again. 
		We split the integral in \eqref{13} instead as follows:
		\begin{equation*}
		\begin{aligned}
		\int_0^\infty e^{\mathrm{i}t\Phi(\xi)}\widehat{g}(t,\xi)\,\diff \xi
		&=\int_0^\infty e^{\mathrm{i}t\Phi(\xi)}\big(1-\psi\big(4(\xi-\xi_0)/\xi_0\big)\big)\widehat{g}(t,\xi)\,\diff \xi\\
		&\quad+\int_0^\infty e^{\mathrm{i}t\Phi(\xi)}\psi\big(4(\xi-\xi_0)/\xi_0\big)\widehat{g}(t,\xi)\,\diff \xi\\
		&=\colon B_3+B_4.
		\end{aligned}
		\end{equation*}
		We first estimate  the term \(B_3\). By the fact that \(|\partial_\xi\Phi|\gtrsim\max(\xi^4,\xi_0^4)\) 
		on the support of the integral, we may estimate \(B_3\) in a similar fashion as \eqref{15.3}-\eqref{15.4} to obtain
		\begin{equation*}
		\begin{aligned} 
		|B_3|\lesssim t^{-1}\xi_0^{-4}+t^{-9/10}\xi_0^{-7/2}\lesssim t^{-1/5}(-x/t^{1/5})^{-7/8},
		\end{aligned}
		\end{equation*}
		which is better than the desired bound \(t^{-1/5}(-x/t^{1/5})^{-9/20}\).

		We next control  the term \(B_4\).  Let \(\tilde{l}_0\) be the smallest integer such that \(2^{\tilde{l}_0}\geq t^{-1/5}(-x/t^{1/5})^{-3/10}\) and bound the term \(B_4\) as follows:
		\begin{equation*}
		\begin{aligned}
		&|B_4|\leq \sum_{l=\tilde{l}_0}^{\log \xi_0+10}|B_{4l}|,\\
		&B_{4\tilde{l}_0}=\int_0^\infty e^{\mathrm{i}t\Phi(\xi)}\psi\left(4(\xi-\xi_0)\xi_0^{-1}\right)\varphi\big(2^{-\tilde{l}_0}(\xi-\xi_0)\big)\widehat{g}(t,\xi)\,\diff \xi,\\
		&B_{4l}=\int_0^\infty e^{\mathrm{i}t\Phi(\xi)}\psi\left(4(\xi-\xi_0)\xi_0^{-1}\right)\psi\big(2^{-l}(\xi-\xi_0)\big)\widehat{g}(t,\xi)\,\diff \xi,\quad l\geq \tilde{l}_0+1.
		\end{aligned}
		\end{equation*}
		Using the fact that \(|\partial_\xi\Phi|\gtrsim 2^l\xi_0^3\) and integrating by parts, similar to \eqref{16}-\eqref{17}, we have
		\begin{equation*}
		\begin{aligned}
		|B_{4l}|\lesssim t^{-1}\left(\|\widehat{g}\|_{L^\infty}\xi_0^{-3}2^{-l}+\|\partial\widehat{g}\|_{L^2}\xi_0^{-3}2^{-l/2}\right),
		 \quad l\geq \tilde{l}_0+1,
		\end{aligned}
		\end{equation*}
   which yields the desired bound \(t^{-1/5}(-x/t^{1/5})^{-9/20}\) after summing in \(l\).

		To study the contributions from the term \(B_{4\tilde{l}_0}\),  we split further 
		\begin{equation*}
		\begin{aligned}
		&B_{4\tilde{l}_0}= e^{\mathrm{i}t\Phi(\xi_0)}\widehat{g}(t,\xi_0)\int_0^\infty e^{10\mathrm{i}t\xi_0^3(\xi-\xi_0)^2}\psi\left(4(\xi-\xi_0)\xi_0^{-1}\right)\varphi\big((\xi-\xi_0)2^{-\tilde{l}_0}\big)\,\diff \xi\\
		&+e^{\mathrm{i}t\Phi(\xi_0)}\int_0^\infty e^{10\mathrm{i}t\xi_0^3(\xi-\xi_0)^2}\psi\left(4(\xi-\xi_0)\xi_0^{-1}\right)
		\varphi\big((\xi-\xi_0)2^{-\tilde{l}_0}\big)\big(\widehat{g}(t,\xi)-\widehat{g}(t,\xi_0)\big)\,\diff \xi\\
		&+\int_0^\infty \big(e^{\mathrm{i}t\Phi(\xi)}-e^{\mathrm{i}t\Phi(\xi_0)+10\mathrm{i}t\xi_0^3(\xi-\xi_0)^2}\big)\psi\left(4(\xi-\xi_0)\xi_0^{-1}\right)\varphi\big((\xi-\xi_0)2^{-\tilde{l}_0}\big)\widehat{g}(t,\xi)\,\diff \xi\\
		&=\colon D_1+D_2+D_3.
		\end{aligned}
		\end{equation*}
In light of the bounds in \eqref{15}, we estimate
		\begin{equation*}
		\begin{aligned}
		|D_2|\lesssim t^{1/10}2^{3\tilde{l}_0/2}\lesssim t^{-1/5}(-x/t^{1/5})^{-9/20},
		\end{aligned}
		\end{equation*}
and
\begin{equation*}
\begin{aligned}
|D_3|\lesssim t2^{4\tilde{l}_0}(\xi_0^2+2^{2\tilde{l}_0})\lesssim t^{-1/5}(-x/t^{1/5})^{-7/10},
\end{aligned}
\end{equation*}		
where the bound is better than what we need.		
		For the term \(D_1\), we write
		\begin{equation*}
		\begin{aligned}
		D_1= e^{\mathrm{i}t\Phi(\xi_0)}\widehat{g}(t,\xi_0)\int_{-\xi_0}^{\infty} e^{10\mathrm{i}t\xi_0^3\eta^2}\psi(4\eta\xi_0^{-1})\varphi(\eta 2^{-\tilde{l}_0})\,\diff \eta.
		\end{aligned}
		\end{equation*}
Recalling the formula
		\begin{equation*}
		\begin{aligned}
		\int_{-\infty}^{\infty} e^{-bx^2}\,\diff x=\sqrt{\frac{\pi}{b}},\quad b\in\mathbb{C},\ \Re b>0,
		\end{aligned}
		\end{equation*}
		one calculates that
		\begin{equation*}
		\begin{aligned}
		\int_{-\infty}^{\infty} e^{10\mathrm{i}t\xi_0^3\xi^2}e^{-\xi^2/2^{\tilde{l}_0}}\,\diff \xi=\sqrt{\frac{\mathrm{i}\pi}{10t\xi_0^3}}+\bigO\big(2^{-\tilde{l}_0}(t\xi_0^3)^{-3/2}\big).
		\end{aligned}
		\end{equation*}
In conclusion, we have obtained  
		\begin{equation*}
		\begin{aligned}
		D_1=\sqrt{\frac{\mathrm{i}\pi}{10t\xi_0^3}}e^{\mathrm{i}t\Phi(\xi_0)}\widehat{f}(t,\xi_0)+\bigO\big(t^{-1/5}(-x/t^{1/5})^{-9/20}\big),
		\end{aligned}
		\end{equation*}
which combines \eqref{13} to finish the proof of the \eqref{12}.

	\end{proof}

\section{Estimates on \(\|u\|_{H^2}\) and \(\|xf\|_{L^2}\)}\label{sec:2}

In this section we will prove the uniform bounds for the energy part in \eqref{9} which may be stated precisely as follows:
\begin{proposition}\label{pr:1} Let \(u\) be a solution of \eqref{eq:main}-\eqref{eq:initial} satisfying the a priori bounds \eqref{8}. Then the following estimates hold true:
	\begin{align}\label{18}
	\|u(t,\cdot)\|_{H^2}\leq C\varepsilon_0,
	\end{align}
and
	\begin{align}\label{19}
	\|xf(t,\cdot)\|_{L^2}\leq C(\varepsilon_0+\varepsilon_1^5)\langle t\rangle^{1/10}.
	\end{align}	
\end{proposition}

\begin{proof} Recall the definition \(f=e^{-t\partial_x^5}u(t)\). The first estimate \eqref{18} is a consequence of the conservation of Mass and Hamiltonian: 
\begin{align*}
\int_{\R} u^2(x)\,\diff x,\quad \int_{\R} \left(\frac{1}{2}\big(|\partial_x|^2u\big)^2-\frac{1}{30}u^6\right)(x)\,\diff x，
\end{align*}	
where we shall control \(\|u\|_{L^6}\) by the two conservations via the interpolation 	
\begin{align*}
\|u\|_{L^6}\lesssim \|u\|_{L^2}^{5/6}\|u_{xx}\|_{L^2}^{1/6}.
\end{align*}

	 We denote by \(S\) the scaling vector field \(S:=1+x\partial_x+5t\partial_t\), and by \(Ih\) the anti-derivative  of \(h\) vanishing at \(-\infty\), i.e., \(Ih(x)=\int_{-\infty}^x h\,\diff y\). Since \(u\) is a solution of \eqref{eq:main}-\eqref{eq:initial}, a direct calculation shows that \(ISu\) satisfies 
	\begin{align*}
	\partial_tISu-\partial_x^5ISu+u^4\partial_xISu=0.
	\end{align*}
	Multiplying the above equation by \(ISu\), integrating in space and integrating by parts yield
\begin{equation*}
	\begin{aligned}
	&\frac{1}{2}\frac{\diff}{\diff t}\|ISu\|_{L^2}^2=-2\int u^3u_x(ISu)^2\,\diff x\\
	&\lesssim \|u^3u_x\|_{L^\infty}\|ISu\|_{L^2}^2
	\lesssim \varepsilon_1^4t^{-1}\|ISu\|_{L^2}^2,
	\end{aligned}
\end{equation*}
where we have used Lemma \ref{le:1} \eqref{10}  by taking \(g\) to be the profile \(f\) of \(u\) in the last inequality.
We combines the above resulting inequality and Gronwall's inequality to obtain
	\begin{align*}
	\|ISu\|_{L^2}\lesssim \varepsilon_0t^{C\varepsilon_1^4}.
	\end{align*}
We introduce another vector field \(\mathcal{J}=x+5t\partial_x^4\). A small calculation yields  
\begin{align*}
\mathcal{J}u=ISu+tu^5,
\end{align*}
which  by applying Lemma \ref{le:1} \eqref{10} again implies
\begin{equation*}
\begin{aligned}
\|xf\|_{L^2}=\|\mathcal{J}u\|_{L^2}\lesssim \varepsilon_0t^{C\varepsilon_1^4}+t\|u\|_{L^{10}}^5\lesssim \varepsilon_0t^{C\varepsilon_1^4}
+\varepsilon_1^5t^{1/10}.
\end{aligned}
\end{equation*}
This completes the proof of \eqref{19}.
	
\end{proof}

\section{Estimate on \(\|\widehat{f}\|_{L^\infty}\)}\label{sec:3}

The main purpose of this section is to show the following uniform bound on the profile \(f\) in Fourier space:    
\begin{proposition}\label{pr:2} Let \(u\) be a solution of \eqref{eq:main}-\eqref{eq:initial} satisfying the a priori bounds \eqref{8}. Then we have
	\begin{equation}\label{20}
	\begin{aligned}
\|\widehat{f}(t,\cdot)\|_{L^\infty}
\leq C(\varepsilon_0+\varepsilon_1^5).
	\end{aligned}
	\end{equation}
\end{proposition}

\subsection{Evolution of \(\widehat{f}\)}
To prove \eqref{20}, we shall study the evolutionary equation of the profile \(f\) in Fourier space. Let
\begin{equation*}
\begin{aligned}
\Psi(\xi,\eta,\sigma,\eta_1,\sigma_1)=\xi^5-(\xi-\eta-\sigma-\eta_1-\sigma_1)^5
-\eta^5-\sigma^5-\eta_1^5-\sigma_1^5.
\end{aligned}
\end{equation*}
Recall \(f(t)=e^{-t\partial_x^5}u(t)\), then from \eqref{eq:main} we see that the profile \(f\) obeys  
\begin{equation}\label{21}
\begin{aligned}
&\partial_t\widehat{f}(t,\xi)=\frac{-\mathrm{i}\xi}{4\pi^2}\int_{\R^4}e^{-\mathrm{i}t\Psi(\xi,\eta,\sigma,\eta_1,\sigma_1)}\widehat{f}(t,\xi-\eta-\sigma-\eta_1-\sigma_1)\widehat{f}(t,\eta)\widehat{f}(t,\sigma)\\
&\quad\quad\quad\quad\quad\times\widehat{f}(t,\eta_1)\widehat{f}(t,\sigma_1)\,\diff\eta\diff\sigma\diff\eta_1\diff\sigma_1.
\end{aligned}
\end{equation}
Given \(j\in\Z\) and \(|\xi|\in (2^j,2^{j+1})\), to separate the stationary phase, similar to \cite{MR3519470}, we decompose \eqref{21} as follows:
\begin{equation}\label{22}
\begin{aligned}
\partial_t\widehat{f}(t,\xi)
&=\sum_{k,l,k_1,l_1}\frac{-\mathrm{i}\xi}{4\pi^2}\int_{\R^4}e^{-\mathrm{i}t\Psi(\xi,\eta,\sigma,\eta_1,\sigma_1)}\widehat{f}(t,\xi-\eta-\sigma-\eta_1-\sigma_1)\widehat{f}(t,\eta)\\
&\quad\times\widehat{f}(t,\sigma)\widehat{f}(t,\eta_1)\widehat{f}(t,\sigma_1)\psi_k(\eta)\psi_l(\sigma)
\psi_{k_1}(\eta_1)\psi_{l_1}(\sigma_1)\,\diff\eta\diff\sigma\diff\eta_1\diff\sigma_1\\
&=\underbrace{\sum_{\substack{2^k,2^l,2^{k_1},2^{l_1}\lesssim 2^j\\ 2^j>t^{-1/5}}}A_{k,l,k_1,l_1}}_{M}+\underbrace{\sum_{\substack{2^k\sim 2^l\sim 2^{k_1}\sim 2^{l_1}\gg 2^j\\ 2^k>t^{-1/5}}}A_{k,l,k_1,l_1}}_{R_1}\\
&\quad+\underbrace{\sum_{\substack{2^k\sim\mathcal{B}\gg \mathcal{A}\setminus\mathcal{B},2^j\\ 2^k>t^{-1/5}}}A_{k,l,k_1,l_1}}_{R_2}
+\underbrace{\sum_{\substack{2^k,2^l,2^{k_1},2^{l_1},2^j\lesssim t^{-1/5}}}A_{k,l,k_1,l_1}}_{R_3}\\
&\quad+\big\{\text{similar terms to}\ R_1,R_2 \big\},
\end{aligned}
\end{equation}
where  \(\mathcal{A}=\{2^l,2^{k_1},2^{l_1}\}\) and \(\mathcal{B}\subseteq\mathcal{A}\). For simplicity, here we slightly abuse notation \(2^k\sim\mathcal{B}\gg \mathcal{A}\setminus\mathcal{B},2^j\) whose exact meaning is that \(R_2\) contains the following \(7\) terms:
\[
\sum_{\substack{2^k\gg 2^l,2^{k_1},2^{l_1},2^j\\ 2^k>t^{-1/5}}}A_{k,l,k_1,l_1}, \sum_{\substack{2^k\sim 2^l\gg 2^{k_1},2^{l_1},2^j\\ 2^k>t^{-1/5}}}A_{k,l,k_1,l_1}, \ \dots, \ \sum_{\substack{2^k\sim 2^{k_1}\sim 2^{l_1}\gg 2^l,2^j\\ 2^k>t^{-1/5}}}A_{k,l,k_1,l_1}.
\]
We put these terms into \(R_2\) since they can be treated in a same manner. 

In this subsection we aim to show the following crucial proposition:	
\begin{proposition}\label{pr:3}	Let \(t>1\). Then the profile \(f\) obeys the following ODE in Fourier space:
	\begin{equation}\label{23}
	\begin{aligned}
	\partial_t\widehat{f}(t,\xi)&=\bigg(\frac{-\mathrm{i}}{40t^2\xi^5}|\widehat{f}(t,\xi)|^4\widehat{f}(t,\xi)+\frac{c_1\mathrm{i}}{t^2\xi^5}e^{-\frac{624\mathrm{i}t\xi^5}{625}}\widehat{f}(t,\xi/5)^5\\
	&\quad+\frac{c_2\mathrm{i}}{t^2\xi^5}e^{-\frac{80\mathrm{i}t\xi^5}{81}}|\widehat{f}(t,\xi/3)|^2\widehat{f}(t,\xi/3)^3\bigg){\bf{1}}_{|\xi|>t^{-1/5}}+R(t,\xi),
	\end{aligned}
	\end{equation}
	where \(c_1\) and \(c_2\) are constants depending on \(\mathrm{sign}\ \xi\), and the remainder \(R(t,\xi)\) is integrable in time and satisfies
	\begin{align}\label{24}
	\int_0^\infty |R(t,\xi)|\,\diff t\lesssim \varepsilon_1^5.
	\end{align}
\end{proposition}

We divide the proof of Proposition \ref{pr:3} into the following four steps.
\bigskip

\noindent{\emph{Step 1: Estimate of \(M\)}}. This is the main case. By change of variables
\begin{align}\label{24.5}
(\xi,\eta,\sigma,\eta_1,\sigma_1)=2^j(\xi^\prime,\eta^\prime,\sigma^\prime,\eta_1^\prime,\sigma_1^\prime),
\end{align} 
we rewrite the term \(M\) into the following form
\begin{equation}\label{25}
\begin{aligned}
M&=-\mathrm{i}(4\pi^2)^{-1}\xi\int_{\R^4}e^{-\mathrm{i}t\Psi(\xi,\eta,\sigma,\eta_1,\sigma_1)}\widehat{f_{\lesssim j}}(t,\xi-\eta-\sigma-\eta_1-\sigma_1)\widehat{f_{\lesssim j}}(t,\eta)\widehat{f_{\lesssim j}}(t,\sigma)\\
&\quad\times\widehat{f_{\lesssim j}}(t,\eta_1)\widehat{f_{\lesssim j}}(t,\sigma_1)\varphi_j(C\eta)\varphi_j(C\sigma)\varphi_j(C\eta_1)
\varphi_j(C\sigma_1)\,\diff\eta\diff\sigma\diff\eta_1\diff\sigma_1\\
&=-\mathrm{i}(4\pi^2)^{-1}2^{5j}\int_{\R^4}e^{-\mathrm{i}t2^{5j}\Psi(\xi^\prime,\eta^\prime,\sigma^\prime,\eta_1^\prime,\sigma_1^\prime)}\xi^\prime\widehat{f_{\lesssim j}}\left(t,2^j(\xi^\prime-\eta^\prime-\sigma^\prime-\eta_1^\prime-\sigma_1^\prime)\right)\\
&\quad\times\widehat{f_{\lesssim j}}(t,2^j\eta^\prime)\widehat{f_{\lesssim j}}(t,2^j\sigma^\prime)
\widehat{f_{\lesssim j}}(t,2^j\eta_1^\prime)\widehat{f_{\lesssim j}}(t,2^j\sigma_1^\prime)
\varphi(C\eta^\prime)\varphi(C\sigma^\prime)\\
&\quad\times\varphi(C\eta_1^\prime)\varphi(C\sigma_1^\prime)\,\diff\eta^\prime\diff\sigma^\prime\diff\eta_1^\prime\diff\sigma_1^\prime,
\end{aligned}
\end{equation}
where \(C\sim 1\).
If we let
\begin{equation}\label{25.5}
\begin{aligned}
F(\eta^\prime,\sigma^\prime,\eta_1^\prime,\sigma_1^\prime)&=\widehat{f_{\lesssim j}}\left(t,2^j(\xi^\prime-\eta^\prime-\sigma^\prime-\eta_1^\prime-\sigma_1^\prime)\right)\widehat{f_{\lesssim j}}(t,2^j\eta^\prime)\\
&\quad\times\widehat{f_{\lesssim j}}(t,2^j\sigma^\prime)\widehat{f_{\lesssim j}}(t,2^j\eta_1^\prime)\widehat{f_{\lesssim j}}(t,2^j\sigma_1^\prime),
\end{aligned}
\end{equation}
then \eqref{25} becomes 
\begin{equation}\label{26}
\begin{aligned}
M
&=-\mathrm{i}(4\pi^2)^{-1}2^{5j}\int_{\R^4}e^{-\mathrm{i}t2^{5j}\Psi(\xi^\prime,\eta^\prime,\sigma^\prime,\eta_1^\prime,\sigma_1^\prime)}\xi^\prime F(\eta^\prime,\sigma^\prime,\eta_1^\prime,\sigma_1^\prime)\\
&\quad\times\varphi(C\eta^\prime)\varphi(C\sigma^\prime)\varphi(C\eta_1^\prime)\varphi(C\sigma_1^\prime)\,\diff\eta^\prime\diff\sigma^\prime\diff\eta_1^\prime\diff\sigma_1^\prime,
\end{aligned}
\end{equation}
where \(|\eta^\prime|,|\sigma^\prime|,|\eta_1^\prime|,|\sigma_1^\prime|\lesssim |\xi^\prime|\sim 1\). 
\bigskip

To analyze \(M\), we need the following stationary phase in 4-dimension:
\begin{lemma}\label{le:2} Assume \(\chi\in C_0^\infty\) and \(|\nabla\chi|+|\nabla^2\chi|\leq 1\); and
	\(\phi\in C^\infty\) such that \(|\det \mathrm{Hess}\ \phi|\geq1\) and 
	\(|\nabla\phi|+|\nabla^2\phi|+|\nabla^3\phi|\leq 1\). Let 
	\begin{align*}
	J=\int_{\R^4}e^{\mathrm{i}\lambda\phi(\xi_1,\xi_2,\xi_3,\xi_4)}F(\xi_1,\xi_2,\xi_3,\xi_4)\chi(\xi_1,\xi_2,\xi_3,\xi_4)\, \diff \xi_1 \diff\xi_2 \diff \xi_3 \diff\xi_4.
	\end{align*}
	\(\mathrm{(i)}\) If \(\nabla\phi\) only vanishes at \(({\xi_1}_0,{\xi_2}_0,{\xi_3}_0,{\xi_4}_0)\), then the following asymptotic expansion holds true: 
	\begin{equation*}
	\begin{aligned}
	J&=\frac{4\pi^2 e^{\mathrm{i}\frac{\pi}{4}\beta}}{\sqrt{\Delta}}\frac{e^{\mathrm{i}\lambda\phi({\xi_1}_0,{\xi_2}_0,{\xi_3}_0,{\xi_4}_0)}}{\lambda^2}F({\xi_1}_0,{\xi_2}_0,{\xi_3}_0,{\xi_4}_0)\chi({\xi_1}_0,{\xi_2}_0,{\xi_3}_0,{\xi_4}_0)\\
	&\quad+\bigO\left(\frac{\|\widehat{F}\|_{L^1}}{\lambda^2}\right),
	\end{aligned}
	\end{equation*}
	where \(\beta=\mathrm{sign}\ \mathrm{Hess}\ \phi({\xi_1}_0,{\xi_2}_0,{\xi_3}_0,{\xi_4}_0)\) and \(\Delta=|\det \mathrm{Hess}\ \phi({\xi_1}_0,{\xi_2}_0,{\xi_3}_0,{\xi_4}_0)|\).\\
	\(\mathrm{(ii)}\) If \(|\nabla\phi|\geq1\), then we have
	\begin{align*}
	J=\bigO\left(\frac{\|\widehat{F}\|_{L^1}}{\lambda^2}\right).
	\end{align*}
	
\end{lemma}

\begin{proof} The results follow from similar arguments of Lemma A.1 in \cite{MR3519470}.    
\end{proof}

A direct computation shows that the stationary points of \(\Psi\)  are given by
\begin{equation*}
\begin{aligned}
(\eta_1,\sigma_1,{\eta_1}_1,{\sigma_1}_1)&=(\xi/5,\xi/5,\xi/5,\xi/5),\\
\bigcup_{a=2}^6(\eta_a,\sigma_a,{\eta_1}_a,{\sigma_1}_a)&=\big\{(-\xi/3,\xi/3,\xi/3,\xi/3),(\xi/3,-\xi/3,\xi/3,\xi/3),\\
&\quad\quad (\xi/3,\xi/3,-\xi/3,\xi/3),(\xi/3,\xi/3,\xi/3,-\xi/3),\\
&\quad\quad (\xi/3,\xi/3,\xi/3,\xi/3)\big\},\\
\bigcup_{a=7}^{16}(\eta_a,\sigma_a,{\eta_1}_a,{\sigma_1}_a)&=\big\{(-\xi,-\xi,\xi,\xi),(-\xi,\xi,-\xi,\xi),(-\xi,\xi,\xi,-\xi),\\
&\quad\quad (\xi,-\xi,-\xi,\xi),
(\xi,-\xi,\xi,-\xi),(\xi,\xi,-\xi,-\xi),\\
&\quad\quad (\xi,\xi,\xi,-\xi),(\xi,\xi,-\xi,\xi),(\xi,-\xi,\xi,\xi),
(-\xi,\xi,\xi,\xi)\big\}.
\end{aligned}
\end{equation*}
Moreover, we have 
\begin{equation}\label{28}
\begin{aligned}
&\Psi(\eta_1,\sigma_1,{\eta_1}_1,{\sigma_1}_1)=(624/625)\xi^5,\quad
\Delta_1=5^{-7}\cdot 4^4|\xi|^{12},\\
&\Psi(\eta_a,\sigma_a,{\eta_1}_a,{\sigma_1}_a)=(80/81)\xi^5,\quad
\Delta_a=3^{-11}\cdot 4^4\cdot 5^4|\xi|^{12},\quad 2\leq a\leq 6,\\
&\Psi(\eta_a,\sigma_a,{\eta_1}_a,{\sigma_1}_a)=0,\quad
\Delta_a=4^4\cdot 5^4|\xi|^{12},\quad 7\leq a\leq 16,
\end{aligned}
\end{equation}
and
\begin{equation}\label{29}
\begin{aligned}
&\beta_a=1-\mathrm{sign}\ \xi,\quad &\text{for}\ 1\leq a\leq 6,\\
&\beta_a=0,\quad\quad &\text{for}\ 7\leq a\leq 16.
\end{aligned}
\end{equation}

Observing \eqref{28}-\eqref{29} and \eqref{24.5}, and applying Lemma \ref{le:2} (i) to \eqref{26}, we obtain 
\begin{equation}\label{30}
\begin{aligned}
M&=\frac{-\mathrm{i}}{40t^2\xi^5}|\widehat{f}(t,\xi)|^4\widehat{f}(t,\xi)+\frac{c_1\mathrm{i}}{t^2\xi^5}e^{-\frac{624\mathrm{i}t\xi^5}{625}}\widehat{f}(t,\xi/5)^5\\
&\quad+\frac{c_2\mathrm{i}}{t^2\xi^5}e^{-\frac{80\mathrm{i}t\xi^5}{81}}|\widehat{f}(t,\xi/3)|^2\widehat{f}(t,\xi/3)^3
+\underbrace{2^{5j}\bigO\left(\frac{\|\widehat{F}\|_{L^1}}{(2^{5j}t)^2}\right)}_{R_0},
\end{aligned}
\end{equation}
 where \(c_1\) and \(c_2\) are constants depending on \(\mathrm{sign}\ \xi\). The rest of this step is to show that \(R_0\) is integrable in time and satisfies 
\begin{equation}\label{31}
\begin{aligned}
\int_{2^{-5j}}^\infty |R_0|\,\diff s\lesssim \epsilon_1^5.
\end{aligned}
\end{equation}

As a prerequisite, we give the following estimate: 
\begin{lemma}\label{le:3} Suppose \(2^j> t^{-1/5}\), then we have
	\begin{equation*}
	\begin{aligned}
	\|f_{\lesssim j}\|_{L^1}
	\lesssim \epsilon_12^{\frac{j}{4}}t^{\frac{1}{20}}.
	\end{aligned}
	\end{equation*}
	
\end{lemma}
\begin{proof}
	We first see that
	\begin{equation*}
	\begin{aligned}
	\|f_a\|_{L^2}\leq \|\psi_a\|_{L^2}\|\widehat{f}\|_{L^\infty}\lesssim \epsilon_12^{a/2},
	\end{aligned}
	\end{equation*}
	and
	\begin{equation*}
	\begin{aligned}
	\|xf_a\|_{L^2}=\|\partial \widehat{f_a}\|_{L^2}
	&\lesssim 2^{-a}\|\psi_a^\prime\|_{L^2}\|\widehat{f}\|_{L^\infty}+\|\psi_a\|_{L^\infty}\|\partial \widehat{f}\|_{L^2}\\
	&\lesssim \epsilon_1(2^{-a/2}+t^{1/10}).
	\end{aligned}
	\end{equation*}
Hence we have
	\begin{equation}\label{32}
	\begin{aligned}
	\|f_a\|_{L^1}\lesssim \|f_a\|_{L^2}^{\frac{1}{2}}\|xf_a\|_{L^2}^{\frac{1}{2}}\lesssim  \epsilon_1\left(1+2^{\frac{a}{4}}t^{\frac{1}{20}}\right).
	\end{aligned}
	\end{equation}
	We next estimate 
	\begin{equation}\label{33}
	\begin{aligned}
	\|f_{\leq-\frac{1}{5}\log t}\|_{L^2}
	\leq \|\varphi(t^{1/5}\xi)\|_{L^2}\|\widehat{f}\|_{L^\infty}\lesssim \epsilon_1 t^{-1/10},
	\end{aligned}
	\end{equation}
	and
	\begin{equation}\label{34}
	\begin{aligned}
	&\quad\|xf_{\leq-\frac{1}{5}\log t}\|_{L^2}=\|\partial_\xi[\varphi(t^{1/5}\xi)\widehat{f}(\xi)]\|_{L^2}\\
	&\leq t^{1/5}\|\varphi^\prime(t^{1/5}\xi)\|_{L^2}\|\widehat{f}\|_{L^\infty}+\|\varphi(t^{1/5}\xi)\|_{L^\infty}\|\partial\widehat{f}\|_{L^2}
	\lesssim \epsilon_1 t^{1/10}.
	\end{aligned}
	\end{equation}
It finally concludes from \eqref{32}-\eqref{34} that  
	\begin{equation*}
	\begin{aligned}
	\|f_{\lesssim j}\|_{L^1}&\leq \|f_{\leq-\frac{1}{5}\log t}\|_{L^1}+\sum_{t^{-1/5}\leq 2^a\lesssim 2^j}\|f_a\|_{L^1}\\
	&\lesssim \epsilon_1t^{-\frac{1}{10}\times\frac{1}{2}}t^{\frac{1}{10}\times\frac{1}{2}}+\epsilon_1\sum_{t^{-1/5}\leq 2^a\lesssim 2^j}\left(1+2^{\frac{a}{4}}t^{\frac{1}{20}}\right)\\
	&\lesssim \epsilon_1+\epsilon_12^{\frac{j}{4}}t^{\frac{1}{20}}
	\lesssim \epsilon_12^{\frac{j}{4}}t^{\frac{1}{20}},
	\end{aligned}
	\end{equation*}
	where we have used \(2^j> t^{-1/5}\) in the last two inequalities.

\end{proof}

From \eqref{25.5} we calculate   
\begin{equation}\label{35}
\begin{aligned}
\widehat{F}(x,y,x_1,y_1)&=\frac{2^{-5j}}{4\pi^2}\int e^{-\mathrm{i}z\xi}f_{\lesssim j}\left(t,2^{-j}(z-x)\right)f_{\lesssim j}\left(t,2^{-j}(z-x_1)\right)\\
&\quad\times f_{\lesssim j}(t,2^{-j}z)f_{\lesssim j}\left(t,2^{-j}(y-z)\right)f_{\lesssim j}\left(t,2^{-j}(y_1-z)\right)\,\diff z,
\end{aligned}
\end{equation}
which gives 
\begin{equation}\label{36}
\begin{aligned}
\|\widehat{F}\|_{L^1}\lesssim \|f_{\lesssim j}\|_{L^1}^5.
\end{aligned}
\end{equation}
Applying Lemma \ref{le:3} to \eqref{36}, one has  
\begin{equation}\label{37}
\begin{aligned}
\left|R_0\right|
\lesssim 2^{-5j}t^{-2}\|f_{\lesssim j}\|_{L^1}^5
\lesssim \epsilon_1^52^{-\frac{15}{4}j}t^{-\frac{7}{4}}.
\end{aligned}
\end{equation}
It follows that 
\begin{equation*}
\begin{aligned}
\int_{2^{-5j}}^\infty |R_0|\,\diff s
\lesssim \epsilon_1^52^{-\frac{15}{4}j}\int_{2^{-5j}}^\infty s^{-\frac{7}{4}}\,\diff s\lesssim \epsilon_1^5.
\end{aligned}
\end{equation*}
This completes the proof of \eqref{31}.

\bigskip

\noindent{\emph{Step 2: Estimate of \(R_1\)}}. There is no stationary point in this case.  
We use change variables 
\((\eta,\sigma,\eta_1,\sigma_1)=2^k(\eta^\prime,\sigma^\prime,\eta_1^\prime,\sigma_1^\prime)\)
to write 
\begin{equation*}
\begin{aligned}
R_1
&=-\sum_{\substack{2^k\gg 2^j\\ 2^k>t^{-1/5}}}\mathrm{i}(4\pi^2)^{-1} 2^{4k}\xi\int_{\R^4}e^{-\mathrm{i}t2^{5k}\Psi(2^{-k}\xi,\eta^\prime,\sigma^\prime,\eta_1^\prime,\sigma_1^\prime)}F_k(\eta^\prime,\sigma^\prime,\eta_1^\prime,\sigma_1^\prime)\\
&\quad\times\psi(\eta^\prime)\psi(\sigma^\prime)\psi(\eta_1^\prime)\psi(\sigma_1^\prime)\,\diff\eta^\prime\diff\sigma^\prime\diff\eta_1^\prime\diff\sigma_1^\prime,
\end{aligned}
\end{equation*}
where 
\begin{equation*}
\begin{aligned}
F_k(\eta^\prime,\sigma^\prime,\eta_1^\prime,\sigma_1^\prime)&=\widehat{f_{\lesssim k}}\left(t,\xi-2^k(\eta^\prime+\sigma^\prime+\eta_1^\prime+\sigma_1^\prime)\right)\widehat{f_{\lesssim k}}(t,2^k\eta^\prime)\\
&\quad\times\widehat{f_{\lesssim k}}(t,2^k\sigma^\prime)\widehat{f_{\lesssim k}}(t,2^k\eta_1^\prime)\widehat{f_{\lesssim k}}(t,2^k\sigma_1^\prime).
\end{aligned}
\end{equation*}
From Lemma \ref{le:2} (ii) it follows that
\begin{equation}\label{38}
\begin{aligned}
|R_1|\lesssim 2^j\sum_{\substack{2^k\gg 2^j\\ 2^k>t^{-1/5}}}2^{4k}\frac{\|\widehat{F_k}\|_{L^1}}{(2^{5k}t)^2}.
\end{aligned}
\end{equation}
Proceeding as \eqref{35}-\eqref{36}, the following estimate holds true:
\begin{equation}\label{39}
\begin{aligned}
\|\widehat{F_k}\|_{L^1}\lesssim \|\widehat{f_{\lesssim k}}\|_{L^1}^5.
\end{aligned}
\end{equation}
Inserting \eqref{39} into \eqref{38}, one then uses  Lemma \ref{le:3} to estimate
\begin{equation}\label{40}
\begin{aligned}
|R_1|&\lesssim 2^j\sum_{\substack{2^k\gg 2^j\\ 2^k>t^{-1/5}}}2^{-6k}t^{-2}\|\widehat{f_{\lesssim k}}\|_{L^1}^5
\lesssim \epsilon_1^52^j\sum_{\substack{2^k\gg 2^j\\ 2^k>t^{-1/5}}}2^{-\frac{19}{4}k}t^{-\frac{7}{4}}\\
&\lesssim \epsilon_1^52^j t^{-\frac{7}{4}}\max\left(2^j,t^{-1/5}\right)^{-\frac{19}{4}}.
\end{aligned}
\end{equation}
This entails that \(R_1\) is integrable in time and satisfies 
\begin{equation}\label{41}
\begin{aligned}
\int_0^\infty|R_1|\,\diff s
\lesssim \epsilon_1^5.
\end{aligned}
\end{equation}

\noindent{\emph{Step 3: Estimate of \(R_2\)}}. Recall that there are \(7\) terms in the summands in \(R_2\), we first handle the case \(2^k\gg 2^l,2^{k_1},2^{l_1},2^j\), i.e., \(\mathcal{B}=\emptyset\), and denote by \(R_{21}\) the corresponding term after summation in \(k\),  and then explain how to treat the other cases by a same manner. We are going to show 
\begin{equation}\label{42}
\begin{aligned}
\int_0^\infty|R_2|\,\diff s
\lesssim \epsilon_1^5.
\end{aligned}
\end{equation}
\\  
\noindent{\emph{Case 1: \(\mathcal{B}=\emptyset\)}}.
Notice that \(|\eta|\) is the largest variable among \(\{\xi,\eta,\sigma,\eta_1,\sigma_1\}\), one thus may rewrite
\begin{equation}\label{43}
\begin{aligned}
R_{21}&=-\sum_{\substack{2^k\gg 2^j\\ 2^k>t^{-1/5}}}\mathrm{i}(4\pi^2)^{-1}\xi\int_{\R^4}e^{-\mathrm{i}t\Psi(\xi,\eta,\sigma,\eta_1,\sigma_1)}\widehat{f_{\sim k}}(t,\xi-\eta-\sigma-\eta_1-\sigma_1)\\
&\quad\times\widehat{f_{\sim k}}(t,\eta)\widehat{f_{\ll k}}(t,\sigma)\widehat{f_{\ll k}}(t,\eta_1)\widehat{f_{\ll k}}(t,\sigma_1)\psi_k(\eta)\varphi_k(\sigma)
\varphi_k(\eta_1)\\
&\quad\times\varphi_k(\sigma_1)\,\diff\eta\diff\sigma
\diff\eta_1\diff\sigma_1.
\end{aligned}
\end{equation}

Let \(m_1,m_2,m_3,m_4\in \{0,\dots,10\}\). On the support of the resulting integrand, one calculates that
\begin{equation}\label{43.1}
\begin{aligned}
|\partial_\sigma \Psi|=5\left|(\xi-\eta-\sigma-\eta_1-\sigma_1)^4-\sigma^4\right|\sim 2^{4k},
\end{aligned}
\end{equation} 
and 
\begin{equation}\label{43.2}
\begin{aligned}
\left|\partial_\eta^{m_1}\partial_\sigma^{m_2}\partial_{\eta_1}^{m_3}\partial_{\sigma_1}^{m_4}\left((\partial_\sigma \Psi)^{-1}\right)\right|\lesssim 2^{-4k}2^{-(m_1+m_2+m_3+m_4)k}.
\end{aligned}
\end{equation}
 Using the identity
\begin{equation*}
\begin{aligned}
e^{-\mathrm{i}t\Psi(\xi,\eta,\sigma,\eta_1,\sigma_1)}=\frac{\mathrm{i}}{t\partial_\sigma\Psi}\partial_\sigma e^{-\mathrm{i}t\Psi(\xi,\eta,\sigma,\eta_1,\sigma_1)},
\end{aligned}
\end{equation*}
and integrating by parts in \(\sigma\) in \eqref{43}, we obtain
\begin{equation}\label{44}
\begin{aligned}
R_{21}=\sum_{\substack{2^k\gg 2^j\\ 2^k>t^{-1/5}}}\left(R_{k21}+R_{k22}+R_{k23}\right),
\end{aligned}
\end{equation}
where the terms under summation are given by 
\begin{equation*}
\begin{aligned}
&R_{k21}:=(4\pi^2 t)^{-1}\xi\int_{\R^4}e^{-\mathrm{i}t\Psi(\xi,\eta,\sigma,\eta_1,\sigma_1)}(\partial_\sigma \Psi)^{-1}\partial_\sigma\widehat{f_{\sim k}}(t,\xi-\eta-\sigma-\eta_1-\sigma_1)\\
&\quad\quad\quad\quad\times\widehat{f_{\sim k}}(t,\eta)\widehat{f_{\ll k}}(t,\sigma)\widehat{f_{\ll k}}(t,\eta_1)\widehat{f_{\ll k}}(t,\sigma_1)\psi_k(\eta)\varphi_k(\sigma)\\
&\quad\quad\quad\quad\times\varphi_k(\eta_1)\varphi_k(\sigma_1)\,\diff\eta\diff\sigma\diff\eta_1\diff\sigma_1,\\
&R_{k22}:=(4\pi^2 t)^{-1}\xi\int_{\R^4}e^{-\mathrm{i}t\Psi(\xi,\eta,\sigma,\eta_1,\sigma_1)}(\partial_\sigma \Psi)^{-1}\widehat{f_{\sim k}}(t,\xi-\eta-\sigma-\eta_1-\sigma_1)\\
&\quad\quad\quad\quad\times\widehat{f_{\sim k}}(t,\eta)\partial_\sigma\widehat{f_{\ll k}}(t,\sigma)\widehat{f_{\ll k}}(t,\eta_1)\widehat{f_{\ll k}}(t,\sigma_1)\psi_k(\eta)\varphi_k(\sigma)\\
&\quad\quad\quad\quad\times\varphi_k(\eta_1)\varphi_k(\sigma_1)\,\diff\eta\diff\sigma\diff\eta_1\diff\sigma_1,\\
&R_{k23}:=(4\pi^2 t)^{-1}\xi\int_{\R^4}e^{-\mathrm{i}t\Psi(\xi,\eta,\sigma,\eta_1,\sigma_1)}\partial_\sigma\left((\partial_\sigma \Psi)^{-1}\varphi_k(\sigma)\right)\\
&\quad\quad\quad\quad\times\widehat{f_{\sim k}}(t,\xi-\eta-\sigma-\eta_1-\sigma_1)\widehat{f_{\sim k}}(t,\eta)\widehat{f_{\ll k}}(t,\sigma)\widehat{f_{\ll k}}(t,\eta_1)\\
&\quad\quad\quad\quad\times\widehat{f_{\ll k}}(t,\sigma_1)\psi_k(\eta)\varphi_k(\eta_1)\varphi_k(\sigma_1)\,\diff\eta\diff\sigma\diff\eta_1\diff\sigma_1.
\end{aligned}
\end{equation*}

To control \(R_{21}\), we need the following estimate on pseudo-product operators satisfying certain strong integrability conditions:
\begin{lemma}\label{le:4} If \(m\in L^1(\R^4)\) satisfies  
	\begin{equation*}
	\begin{aligned}
	\left\|\int_{\R^4}m(\eta,\sigma,\eta_1,\sigma_1)e^{i(x\eta+y\sigma+x_1\eta_1+y_1\sigma_1)}\,\diff\eta\diff\sigma\diff\eta_1\diff\sigma_1\right\|_{L^1_{x,y,x_1,y_1}}:=\|m\|_{S^\infty}<\infty,
	\end{aligned}
	\end{equation*}
then the following estimate holds:	
\begin{equation*}
\begin{aligned}
&\bigg|\int_{\R^4} m(\eta,\sigma,\eta_1,\sigma_1)\widehat{g_1}(-\eta-\sigma-\eta_1-\sigma_1)\widehat{g_2}(\eta)\widehat{g_3}(\sigma)\widehat{g_4}(\eta_1)
\widehat{g_5}(\sigma_1)\,\diff\eta\diff\sigma\diff\eta_1\diff\sigma_1\bigg|\\
&\lesssim \|m\|_{S^\infty}\|g_1\|_{L^{p_1}}\|g_2\|_{L^{p_2}}\|g_3\|_{L^{p_3}}\|g_4\|_{L^{p_4}}\|g_5\|_{L^{p_4}},
\end{aligned}
\end{equation*}	
for any \(p_1,p_2,p_3,p_4,p_5\in[1,\infty]\) satisfying \(\sum_{a=1}^5p_a^{-1}=1\).	
\end{lemma}
\begin{proof} The same argument of Lemma A.2 in \cite{MR3519470} applies here.    
\end{proof}

Corresponding to the estimate in Lemma \ref{le:3}, we also have:    
\begin{lemma}\label{le:5} Let \(2^k> t^{-1/5}\) and \(\delta>0\), then the following estimates hold
	\begin{equation*}
	\begin{aligned}
	\|u_{\ll k}\|_{L^\infty}
	\lesssim \epsilon_12^k,\quad
	\|f_{\ll k}\|_{L^2}
	\lesssim \epsilon_12^{k/2},\quad
	\|\partial\widehat{f_{\ll k}}\|_{L^2}
	\lesssim \epsilon_1t^{\frac{1}{10}+\frac{\delta}{5}}2^{\delta k}.
	\end{aligned}
	\end{equation*}
	
\end{lemma}
\begin{proof} The above three inequalities can be proven in a similar fashion, so we only show the last one:
\begin{equation*}
\begin{aligned}
\|\partial\widehat{f_{\ll k}}\|_{L^2}&\leq \|\partial \mathcal{F}(f_{<-\frac{1}{5}\log t})\|_{L^2}+\sum_{t^{-1/5}\leq 2^a\ll 2^k}\|\partial\widehat{f_a}\|_{L^2}\\
&\lesssim \epsilon_1t^{1/10}+\epsilon_1\sum_{t^{-1/5}\leq 2^a\ll 2^k}(2^{-a/2}+t^{1/10})\\
&\lesssim \epsilon_1t^{1/10}+\epsilon_1t^{1/10}\sum_{t^{-1/5}\leq 2^a\ll 2^k}1\\
&\lesssim \epsilon_1t^{1/10}+\epsilon_1t^{1/10}t^{\delta/5}\sum_{t^{-1/5}\leq 2^a\ll 2^k}2^{\delta a}\\
&\lesssim \epsilon_1t^{\frac{1}{10}+\frac{\delta}{5}}2^{\delta k},
\end{aligned}
\end{equation*}
for any \(\delta>0\).
\end{proof}

Applying Lemma \ref{le:4} and Lemma \ref{le:5}, we obtain
\begin{equation}\label{45}
\begin{aligned}
|R_{k21}|
&\lesssim t^{-1}2^j2^{-4k}\|\partial\widehat{f_{\sim k}}\|_{L^2}\|f_{\ll k}\|_{L^2}\|u_{\ll k}\|_{L^\infty}^2\|u_{\sim k}\|_{L^\infty}\\
&\lesssim \epsilon_1^52^jt^{-11/10}2^{-3k/2},
\end{aligned}
\end{equation}
\begin{equation}\label{46}
\begin{aligned}
|R_{k22}|
&\lesssim t^{-1}2^j2^{-4k}\|f_{\sim k}\|_{L^2}\|\partial\widehat{f_{\ll k}}\|_{L^2}\|u_{\ll k}\|_{L^\infty}^2\|u_{\sim k}\|_{L^\infty}\\
&\lesssim \epsilon_1^52^jt^{-\frac{11}{10}+\frac{\delta}{5}}2^{(-\frac{3}{2}+\delta) k},
\end{aligned}
\end{equation}
and
\begin{equation}\label{47}
\begin{aligned}
|R_{k23}|
&\lesssim t^{-1}2^j2^{-5k}\|f_{\sim k}\|_{L^2}\|f_{\ll k}\|_{L^2}\|u_{\ll k}\|_{L^\infty}^2\|u_{\sim k}\|_{L^\infty}\\
&\lesssim \epsilon_1^52^jt^{-6/5}2^{-2k}.
\end{aligned}
\end{equation}
Choosing \(\delta\in(0,1/2)\), substituting \eqref{45}-\eqref{47} into \eqref{44}, and summing over \(k\), we integrate in time to estimate 
\begin{equation}\label{47.5}
\begin{aligned}
\int_0^\infty |R_{21}|\,\diff s
&\lesssim \epsilon_1^52^j \int_{2^{-5k}}^\infty s^{-11/10}\max\left(2^j,s^{-1/5}\right)^{-3/2}\,\diff s\\
&\quad+ \epsilon_1^52^j \int_{2^{-5k}}^\infty s^{-\frac{11}{10}+\frac{\delta}{5}}\max\left(2^j,s^{-1/5}\right)^{-\frac{3}{2}+\delta}\,\diff s\\
&\quad+ \epsilon_1^52^j \int_{2^{-5k}}^\infty s^{-6/5}\max\left(2^j,s^{-1/5}\right)^{-2}\,\diff s\\
&\lesssim \epsilon_1^5.
\end{aligned}
\end{equation}

\bigskip
\noindent{\emph{Case 2: \(\mathcal{B}\) is a singleton}}.
 Due to symmetry (three cases totally), we only analyze the case \(2^k\sim 2^l\gg 2^{k_1},2^{l_1},2^j\), and denote by \(R_{22}\) the corresponding term after summation in \(k\), which may be written as
\begin{equation*}
\begin{aligned}
R_{22}&=-\sum_{\substack{2^k\gg 2^j\\ 2^k>t^{-1/5}}}\mathrm{i}(4\pi^2)^{-1}\xi\int_{\R^4}e^{-\mathrm{i}t\Psi(\xi,\eta,\sigma,\eta_1,\sigma_1)}\widehat{f_{\sim k}}(t,\xi-\eta-\sigma-\eta_1-\sigma_1)\\
&\quad\times\widehat{f_{\sim k}}(t,\eta)\widehat{f_{\sim k}}(t,\sigma)\widehat{f_{\ll k}}(t,\eta_1)\widehat{f_{\ll k}}(t,\sigma_1)\psi_k(\eta)\psi_k(\sigma)\\
&\quad\times\varphi_k(\eta_1)\varphi_k(\sigma_1)\,\diff\eta\diff\sigma
\diff\eta_1\diff\sigma_1.
\end{aligned}
\end{equation*}

There are two cases: \(k=l\) and \(|k-l|\geq 1\). For the former,  \eqref{43.1} and \eqref{43.2} still holds, and thus one can also apply the argument of \(R_{21}\) to deduce the desired bound.
As for the latter, we have    
\begin{equation*}
\begin{aligned}
|\partial_{\eta_1} \Psi|\sim 2^{4k},\quad \left|\partial_\eta^{m_1}\partial_\sigma^{m_2}\partial_{\eta_1}^{m_3}\partial_{\sigma_1}^{m_4}\left((\partial_{\eta_1} \Psi)^{-1}\right)\right|\lesssim 2^{-4k}2^{-(m_1+m_2+m_3+m_4)k}.
\end{aligned}
\end{equation*} 
This is enough to prove the desired estimate in a similar fashion as \(R_{21}\) after integration by parts in \(\eta_1\) instead.

\bigskip
\noindent{\emph{Case 3: \(\mathcal{B}\) is a  binary set}}.
By symmetry (three cases totally), we only also consider the case \(2^k\sim 2^l\sim 2^{k_1}\gg 2^{l_1},2^j\), and denote by \(R_{23}\) the corresponding term after summation in \(k\), which reads as
\begin{equation*}
\begin{aligned}
R_{23}&=-\sum_{\substack{2^k\gg 2^j\\ 2^k>t^{-1/5}}}\mathrm{i}(4\pi^2)^{-1}\xi\int_{\R^4}e^{-\mathrm{i}t\Psi(\xi,\eta,\sigma,\eta_1,\sigma_1)}\widehat{f_{\sim k}}(t,\xi-\eta-\sigma-\eta_1-\sigma_1)\\
&\quad\times\widehat{f_{\sim k}}(t,\eta)\widehat{f_{\sim k}}(t,\sigma)\widehat{f_{\sim k}}(t,\eta_1)\widehat{f_{\ll k}}(t,\sigma_1)\psi_k(\eta)\psi_k(\sigma)\\
&\quad\times\psi_k(\eta_1)\varphi_k(\sigma_1)\,\diff\eta\diff\sigma
\diff\eta_1\diff\sigma_1.
\end{aligned}
\end{equation*}

We shall discuss \(k=l=l_1\) and \(\max(|k-l|,|k-k_1|,|l-k_1|)\geq 1\), separately. Considering the former, it holds that
\begin{equation*}
\begin{aligned}
|\partial_{\sigma_1} \Psi|\sim 2^{4k},\quad \left|\partial_\eta^{m_1}\partial_\sigma^{m_2}\partial_{\eta_1}^{m_3}\partial_{\sigma_1}^{m_4}\left((\partial_{\sigma_1} \Psi)^{-1}\right)\right|\lesssim 2^{-4k}2^{-(m_1+m_2+m_3+m_4)k},
\end{aligned}
\end{equation*}
which suffices to yield the desired bound via integration by parts in \(\sigma_1\) as \(R_{21}\). 
For the latter, we further divide two sub cases: either there are only two numbers no less than one, or all of the three numbers are no less than one. By symmetry again,  we only consider  
\(|k-l|, |k-k_1|\geq 1, l=k_1\) for the first sub case,
and  may assume \(k>l>k_1\) for the second sub case. Both cases satisfies \eqref{43.1} and \eqref{43.2}  
which give the desired bound by repeating the argument of \(R_{21}\). 

We finish the proof of \eqref{42} by summarizing the above three cases.

\bigskip
\noindent{\emph{Step 4: Estimate of \(R_3\)}}. This is the easiest case. By Young's inequality,  we obtain
\begin{equation*}
\begin{aligned}
|R_3|&\lesssim 2^j{\bf{1}}_{t\lesssim 2^{-5j}}\sum_{\substack{2^k,2^l,2^{k_1},2^{l_1}\lesssim t^{-1/5}}}2^k2^l2^{k_1}2^{l_1}\big\|\mathcal{F}(f_{\lesssim t^{-1/5}})\big\|_{L^\infty}^5\\
&\lesssim \epsilon_1^52^jt^{-4/5}{\bf{1}}_{t\lesssim 2^{-5j}}.
\end{aligned}
\end{equation*}
Therefore \(R_3\) is integrable in time and satisfies
\begin{equation*}
\begin{aligned}
\int_0^\infty |R_3|\,\diff s
\lesssim \epsilon_1^52^j\int_0^{2^{-5j}} s^{-4/5}\,\diff s
\lesssim \epsilon_1^5.
\end{aligned}
\end{equation*}

\qed

\subsection{Proof of Proposition \ref{pr:2}}
To bound \(\|\widehat{f}(t,\cdot)\|_{L^\infty}\), we introduce the modified profile \(w\) as follows:
	\begin{equation}\label{48}
	\begin{aligned}
	\tilde{w}(t,\xi):=e^{\mathrm{i}B(t,\xi)}\widehat{f}(t,\xi)
	\end{aligned}
	\end{equation}
with
	\begin{equation}\label{49}
	\begin{aligned}
	B(t,\xi):=\frac{1}{40\xi^5} \int_1^t|\widehat{f}(s,\xi)|^4\frac{\diff s}{s^2}.
	\end{aligned}
	\end{equation}
In view of \eqref{23} and \eqref{48}-\eqref{49}, we then deduce  
	\begin{equation}\label{50}
	\begin{aligned}
	&\quad\partial_t\tilde{w}(t,\xi)=e^{\mathrm{i}B(t,\xi)}\left(\partial_t\widehat{f}(t,\xi)+\mathrm{i}\partial_tB(t,\xi)\widehat{f}(t,\xi)\right)\\
	&=e^{\mathrm{i}B(t,\xi)}\bigg[\bigg(\frac{c_1\mathrm{i}}{t^2\xi^5}e^{-\frac{624\mathrm{i}t\xi^5}{625}}\widehat{f}(t,\xi/5)^5+\frac{c_2\mathrm{i}}{t^2\xi^5}e^{-\frac{80\mathrm{i}t\xi^5}{81}}|\widehat{f}(t,\xi/3)|^2\widehat{f}(t,\xi/3)^3\bigg){\bf{1}}_{|\xi|>t^{-1/5}}\\
	&\quad\quad\quad\quad\quad\quad+R(t,\xi)\bigg],\quad \text{for}\ t>1.
	\end{aligned}
	\end{equation}
Observing that \(B\) is real, integrating in time in \eqref{50}, and taking account of \eqref{24},
we get
	\begin{equation*}
	\begin{aligned}
	|\widehat{f}(t,\xi)|&\lesssim |\widehat{u_0}(\xi)|+\left|\xi^{-5}\int_{|\xi|^{-5}}^te^{\mathrm{i}B(s,\xi)}e^{-\frac{624\mathrm{i}s\xi^5}{625}}\widehat{f}(s,\xi/5)^5s^{-2}\,\diff s\right|\\
	&\quad+\left|\xi^{-5}\int_{|\xi|^{-5}}^te^{\mathrm{i}B(s,\xi)}e^{-\frac{80\mathrm{i}s\xi^5}{81}}|\widehat{f}(t,\xi/3)|^2\widehat{f}(t,\xi/3)^3s^{-2}\,\diff s\right|+\varepsilon_1^5.
	\end{aligned}
	\end{equation*}
To complete the proof we need to show that
	\begin{equation}\label{51}
	\begin{aligned}
	&\left|\xi^{-5}\int_{|\xi|^{-5}}^te^{\mathrm{i}B(s,\xi)}e^{-\frac{624\mathrm{i}s\xi^5}{625}}\widehat{f}(s,\xi/5)^5s^{-2}\,\diff s\right|\\
	&+\left|\xi^{-5}\int_{|\xi|^{-5}}^te^{\mathrm{i}B(s,\xi)}e^{-\frac{80\mathrm{i}s\xi^5}{81}}|\widehat{f}(t,\xi/3)|^2\widehat{f}(t,\xi/3)^3s^{-2}\,\diff s\right|\lesssim \varepsilon_1^5.
	\end{aligned}
	\end{equation}

	We only show that the first part of \eqref{51} may be bounded by \(\varepsilon_1^5\), the other part can be handled in the same way.
	Observing that
	\[e^{-\frac{624\mathrm{i}s\xi^5}{625}}=-\frac{625}{624\mathrm{i}\xi^5}\partial_se^{-\frac{624\mathrm{i}s\xi^5}{625}},\] 
	and integrating by parts in \(s\), we can bound
	\begin{equation*}
	\begin{aligned}
	\left|\xi^{-5}\int_{|\xi|^{-5}}^te^{\mathrm{i}B(s,\xi)}e^{-\frac{624\mathrm{i}s\xi^5}{625}}\widehat{f}(s,\xi/5)^5s^{-2}\,\diff s\right|\lesssim E_1+E_2+E_3+E_4
	\end{aligned}
	\end{equation*}
with
\begin{equation*}
\begin{aligned}
&E_1=|\xi|^{-10}|\widehat{f}(s,\xi/5)|^5s^{-2}\big|_{s=|\xi|^{-5}}^{s=t},\\
&E_2=|\xi|^{-10}\int_{|\xi|^{-5}}^t|\partial_s\widehat{f}(s,\xi/5)||\widehat{f}(s,\xi/5)|^4s^{-2}\,\diff s,\\
&E_3=|\xi|^{-10}\int_{|\xi|^{-5}}^t|\partial_sB(s,\xi)||\widehat{f}(s,\xi/5)|^5s^{-2}\,\diff s,\\
&E_4=|\xi|^{-10}\int_{|\xi|^{-5}}^t|\widehat{f}(s,\xi/5)|^5s^{-3}\,\diff s.
\end{aligned}
\end{equation*}
Applying the a priori assumption \(\|\widehat{f}(t,\cdot)\|_{L^\infty}\leq \varepsilon_1 \), the desired bound \(\varepsilon_1^5\) for \(E_1\) follows immediately, and one may also estimate the other ones as follows:
	\begin{equation*}
	\begin{aligned}
	E_2\lesssim|\xi|^{-10}\int_{|\xi|^{-5}}^t\left(\varepsilon_1^5|\xi|^{-5}s^{-2}+R(s,\xi)\right)\varepsilon_1^4s^{-2}\,\diff s\lesssim \varepsilon_1^9,
	\end{aligned}
	\end{equation*}
	\begin{equation*}
	\begin{aligned}
	E_3\lesssim |\xi|^{-10}\int_{|\xi|^{-5}}^t\varepsilon_1^4|\xi|^{-5}s^{-2}\varepsilon_1^5s^{-2}\,\diff s\lesssim \varepsilon_1^9,
	\end{aligned}
	\end{equation*}
and
	\begin{equation*}
	\begin{aligned}
	E_4\lesssim |\xi|^{-10}\int_{|\xi|^{-5}}^t\varepsilon_1^5s^{-3}\,\diff s\lesssim \varepsilon_1^5.
	\end{aligned}
	\end{equation*}
The bounds in \(E_2\) and \(E_3\) are stronger than the desired ones since \(\varepsilon_1\in(0,1)\).

\section{Asymptotics}\label{sec:4}

In Section \ref{sec:2}-\ref{sec:3}, we have shown \eqref{9}. Choosing \(\varepsilon_1=\varepsilon_0^{1/5}\), then \eqref{9} becomes 
\begin{align}\label{52}
\|u\|_{X}\leq 2C\varepsilon_0.
\end{align}
The asymptotics \eqref{3} in decaying region is a consequence of \eqref{52},  and  \eqref{10} by taking \(g\) to be the profile \(f\). So this section  is devoted to studying the asymptotics in self-similar region and oscillatory region.

\subsection{Asymptotics in self-similar region}
We introduce the following self-similar change of variables	of the solution \(u\) to \eqref{eq:main}-\eqref{eq:initial}:
	\begin{equation}\label{52.5}
	\begin{aligned}
	v(t,x)=t^{1/5}u(t,xt^{1/5}),
	\end{aligned}
	\end{equation}
and then calculate that 				
	\begin{equation}\label{53}
	\begin{aligned}
	\partial_tv=t^{-1}\partial_x(5^{-1}xv-\partial_x^4v-v^5).
	\end{aligned}
	\end{equation}

Recalling \(\gamma=\frac{1}{5}(\frac{1}{10}-C\varepsilon_1^2)\),  we will show that as \(|x|\leq t^{4\gamma}\) the following estimates hold:
	\begin{equation}\label{54}
	\begin{aligned}
	|P_{\geq 2^{20}t^{\gamma}}v(t,x)|\lesssim \varepsilon_0 t^{-7\gamma/2},
	\end{aligned}
	\end{equation}
and
\begin{equation}\label{55}
\begin{aligned}
|\partial_tP_{\leq 2^{20}t^{\gamma}}v(t,x)|\lesssim \varepsilon_0 t^{-\frac{11}{10}+\frac{3\gamma}{2}+C\epsilon_1^2}.
\end{aligned}
\end{equation}

	To estimate \eqref{54}, we first write 
	\begin{equation*}
	\begin{aligned}
	P_{\geq 2^{20}t^{\gamma}}v(t,x)&=t^{1/5}\int_{-\infty}^\infty e^{\mathrm{i}\tilde{\Phi}(\xi)}\big(1-\varphi(\xi t^{\frac{1}{5}-\gamma}2^{-20})\big)\widehat{f}(t,\xi)\,\diff\xi,\\ 
	\tilde{\Phi}(\xi)&=\tilde{\Phi}(\xi;x,t):=x\xi t^{1/5}+t\xi^5.
	\end{aligned}
	\end{equation*}
Observing that \(t\xi^4\gg |x|t^{1/5}\) on the support of the above integral, so that
\(|\partial_\xi\tilde{\Phi}|\gtrsim t\xi^4\gtrsim t^{\frac{1}{5}+4\gamma}\), we then use integration by parts to bound
	\begin{equation}\label{56}
	\begin{aligned}
	&|P_{\geq 2^{20}t^{\gamma}}v(t,x)|\leq G_1+G_2,\\
	&G_1\lesssim t^{1/5}\int_{-\infty}^\infty \left|\partial_\xi\left[(\partial_\xi\tilde{\Phi})^{-1}\big(1-\varphi(\xi t^{\frac{1}{5}-\gamma}2^{-20})\big)\right]\widehat{f}(t,\xi)\right|\,\diff\xi,\\
	&G_2\lesssim t^{1/5}\int_{-\infty}^\infty \left|(\partial_\xi\tilde{\Phi})^{-1}\big(1-\varphi(\xi t^{\frac{1}{5}-\gamma}2^{-20})\big)\partial_\xi\widehat{f}(t,\xi)\right|\,\diff\xi.
	\end{aligned}
	\end{equation}
Using the bounds on \(\|\widehat{f}\|_{L^\infty}\) and \(\|\partial\widehat{f}\|_{L^2}\) in \eqref{52},  we respectively estimate 
\begin{equation*}
\begin{aligned}
G_1&\lesssim t^{1/5}\|\widehat{f}\|_{L^\infty}\int_{-\infty}^\infty \bigg(t^{-1}|\xi|^{-5}\big|1-\varphi(\xi t^{\frac{1}{5}-\gamma}2^{-20})\big|\\
&\quad\quad\quad\quad\quad\quad+t^{-1}\xi^{-4}\big|\varphi^\prime(\xi t^{\frac{1}{5}-\gamma}2^{-20})t^{\frac{1}{5}-\gamma}\big|\bigg)\,\diff \xi\\
&\lesssim  \varepsilon_0t^{-4\gamma},
\end{aligned}
\end{equation*}
and 
\begin{equation*}
\begin{aligned}
G_2&\lesssim t^{1/5}\|\partial\widehat{f}\|_{L^2}\left(\int_{-\infty}^\infty t^{-2}\xi^{-8}\big|1-\varphi(\xi t^{\frac{1}{5}-\gamma}2^{-20})\big|^2\,\diff \xi\right)^{1/2}\\
&\lesssim \varepsilon_0t^{-7\gamma/2}.
\end{aligned}
\end{equation*}
These two estimates together complete the proof of \eqref{54}. 
	
	We now turn to the proof of \eqref{55}. One first computes 
	\begin{equation*}
	\begin{aligned}
	&\partial_tP_{\leq 2^{20}t^{\gamma}}v(t,x)
	&=\mathcal{F}^{-1}\left(\partial_t\varphi(\xi t^{-\gamma}2^{-20})\widehat{v}(t,\xi)\right)+P_{\leq 2^{20}t^{\gamma}}\partial_tv(t,x).
	\end{aligned}
	\end{equation*}	
We rewrite the first term in the form
	\begin{equation*}
	\begin{aligned}
	\mathcal{F}^{-1}\left(\partial_t\varphi(\xi t^{-\gamma}2^{-20})\widehat{v}(t,\xi)\right)
	&=t^{1/5}\int_{-\infty}^\infty e^{\mathrm{i}\tilde{\Phi}(\xi;x,t)}(\xi t^{\frac{1}{5}-\gamma-1}2^{-20})\\
	&\quad\times\varphi^{\prime}(\xi t^{1/5-\gamma}2^{-20})\widehat{f}(t,\xi)\,\diff\xi,
	\end{aligned}
	\end{equation*}
and  then by a similar fashion to \eqref{56} estimate 
	\begin{equation}\label{56.5}
	\begin{aligned}
	\|\mathcal{F}^{-1}\left(\partial_t\varphi(\xi t^{-\gamma}2^{-20})\widehat{v}(t,\xi)\right)\|_{L^\infty}\lesssim\varepsilon_0t^{-1-\frac{7\gamma}{2}}.
	\end{aligned}
	\end{equation}
For the second term, we write 
	\begin{equation*}
	\begin{aligned}
	P_{\leq 2^{20}t^{\gamma}}\partial_tv(t,x)=\int_{-\infty}^\infty e^{\mathrm{i}x\xi}\varphi(\xi t^{-\gamma}2^{-20})(\mathrm{i}\xi)\mathcal{F}(\partial_x^{-1}\partial_tv)(t,\xi)\,\diff\xi.
	\end{aligned}
	\end{equation*}
Recalling that \(S=1+x\partial_x+5t\partial_t\), it is straightforward to show	
	\begin{equation}\label{57}
	\begin{aligned}
	\partial_tv(t,x)=t^{-1}Su(t,xt^{1/5}).
	\end{aligned}
	\end{equation}
Therefore we may estimate 
	\begin{equation}\label{58}
	\begin{aligned}
	|P_{\leq 2^{20}t^{\gamma}}\partial_tv(t,x)|&\lesssim \|\partial_x^{-1}\partial_tv\|_{L^2}\left(\int_{-\infty}^\infty |\varphi(\xi t^{-\gamma}2^{-20})\xi|^2\,\diff\xi\right)^{1/2}\\
	&\lesssim t^{3\gamma/2}\|\partial_x^{-1}\partial_tv\|_{L^2}\lesssim t^{\frac{3\gamma}{2}-1}\|ISu\|_{L^2}t^{-1/10}\\
	&\lesssim \varepsilon_0 t^{-\frac{11}{10}+\frac{3\gamma}{2}+C\epsilon_1^2}.
	\end{aligned}
	\end{equation}
The desired estimate is a consequence of \eqref{56.5} and \eqref{58}.

	We decompose \(v\) into the following form:
	\begin{equation}
	\begin{aligned}
	v(t,x)&=v(t,x)\left(1-\psi(x/t^{4\gamma})\right)+P_{\geq 2^{20}t^{\gamma}}v(t,x)\psi(x/t^{4\gamma})\\
	&\quad+P_{\leq 2^{20}t^{\gamma}}v(t,x)\psi(x/t^{4\gamma}).
	\end{aligned}
	\end{equation}
Via the definition \eqref{52.5} and the decay estimate \eqref{11}, we have 
	\begin{equation*}
	\begin{aligned}
	\left|v(t,x)\left(1-\psi(x/t^{4\gamma})\right)\right|
	\lesssim \varepsilon_0 t^{-3\gamma/2}. 
	\end{aligned}
	\end{equation*}
This together with \eqref{54}-\eqref{55} implies that \(v(t,x)\) is a Cauchy sequence in time in \(L^\infty\)-norm.
	Let 
	\begin{equation}\label{59}
	\begin{aligned}
	Q(x):=\lim_{t\rightarrow\infty} v(t,x).
	\end{aligned}
	\end{equation}
Thus \eqref{6} immediately follows.
For \(|x|\leq t^{4\gamma}\), from \eqref{54}, \eqref{55} and \eqref{59}, it follows that
	\begin{equation*}
	\begin{aligned}
	|v(t,x)-Q(x)|&\lesssim \varepsilon_0 t^{-7\gamma/2}+\varepsilon_0\int_t^\infty s^{-\frac{11}{10}+\frac{3\gamma}{2}+C\epsilon_1^2}\,\diff s\\
	&\lesssim \varepsilon_0 t^{-7\gamma/2},
	\end{aligned}
	\end{equation*}
which implies \eqref{4}.

It remains to verify that \(Q\) satisfies the ODE \eqref{5}. In view of \eqref{53} and \eqref{57}, we have  
	\begin{equation}\label{60}
	\begin{aligned}
	\|\partial_x^4v-5^{-1}xv+v^5\|_{L^2}=t^{-1/10}\|ISu\|_{L^2}\lesssim \varepsilon_0 t^{-\frac{1}{10}+C\epsilon_1^2}.
	\end{aligned}
	\end{equation}
This together with \eqref{59} completes the proof of \eqref{5}.

\subsection{Asymptotics in oscillatory region}
	This subsection is devoted to determining the leading order asymptotic term for the solution \(u\) of \eqref{eq:main}-\eqref{eq:initial} in the oscillatory region which can be stated precisely as follows:
	\begin{lemma}\label{le:6} Assume \(t\gg 1\), \(|\xi|\geq 2^{-10}t^{-\frac{1}{5}+\gamma}\) and \(\delta\in(0,1/2)\). Let  \(w\) be the modified profile defined by \eqref{48}-\eqref{49},
then there exists \(w_\infty\in L^\infty\) such that
		\begin{equation}\label{61}
		\begin{aligned}
		|\tilde{w}(t,\xi)-w_\infty(\xi)|\lesssim \varepsilon_0(|\xi|t^{1/5})^{-(\frac{1}{2}-\delta)},
		\end{aligned}
		\end{equation}
Moreover, there exits \(f_\infty\in L^\infty\) such that
		\begin{equation}\label{62}
		\begin{aligned}
		\left|\widehat{f}(t,\xi)- \exp\left(\frac{\mathrm{i}}{40t\xi^5}|f_\infty(\xi_0)|^4\right)f_\infty(\xi)\right|
		\lesssim \varepsilon_0 (|\xi|t^{1/5})^{-(\frac{1}{2}-\delta)}.
		\end{aligned}
		\end{equation}
		
	\end{lemma}
\bigskip
We postpone proving Lemma \ref{le:6}, and first show how to use Lemma \ref{le:6} to complete the proof of \eqref{7}.
	\begin{proof}[Proof of \eqref{7}] In the oscillatory region, we first see that 
		\[\xi_0=\sqrt[4]{-x/(5t)}=5^{-1/4}t^{-1/5}(-x/t^{1/5})^{1/4}\geq 5^{-1/4}t^{-\frac{1}{5}+\gamma}.\] 
Recall that \(u(t)=e^{t\partial_x^5}f(t)\),
we then insert \eqref{62} with \(\xi=\xi_0\) into \eqref{12} with \(g=f\) to obtain 
	\begin{equation*}
	\begin{aligned}
	&\left|u(x,t)-(5t\xi_0^3)^{-1/2}\Re\left\{ \exp\left(-4\mathrm{i}t\xi_0^5+\frac{\mathrm{i}\pi}{4}+\frac{1}{40t\xi^5}|f_\infty(\xi_0)|^4\right)f_\infty(\xi_0)\right\}\right|\\
	&\lesssim  \varepsilon_0 (t\xi_0^3)^{-1/2}(t^{1/5}\xi_0)^{-(\frac{1}{2}-\delta)}+\varepsilon_0t^{-1/5}(-x/t^{1/5})^{-9/20}\\
	&\lesssim  \varepsilon_0t^{-1/5}(-x/t^{1/5})^{-\frac{1}{2}+\frac{\delta}{4}}+\varepsilon_0t^{-1/5}(-x/t^{1/5})^{-9/20}\\
	&\lesssim  \varepsilon_0t^{-1/5}(-x/t^{1/5})^{-9/20},
	\end{aligned}
	\end{equation*}
where we have chosen \(\delta>0\) sufficiently small in the last inequality. 
This completes the proof of \eqref{7}.
\end{proof}

We are now going to prove Lemma \ref{le:6}.
	\begin{proof}[Proof of Lemma \ref{le:6}] Given \(t_2\geq t_1\gg 1\) and \(|\xi|\in (2^j,2^{j+1})\) with \(j\in\Z\) such that \(2^j\geq t_1^{-1/5+\gamma}\). For \eqref{61}, we only need to show 
		\begin{equation}\label{62.5}
		\begin{aligned}
		|\tilde{w}(t_1,\xi)-\tilde{w}(t_2,\xi)|\lesssim \varepsilon_0(2^jt_1^{1/5})^{-(\frac{1}{2}-\delta)}.
		\end{aligned}
		\end{equation}
Go back to the decomposition \eqref{22}, and notice that \(|\xi|\geq 2^{-10}t^{-\frac{1}{5}+\gamma}\gg t^{-1/5}\), we then utilize Proposition \ref{pr:3} to obtain 
	\begin{equation*}
	\begin{aligned}
	\partial_t\widehat{f}(t,\xi)&=\frac{-\mathrm{i}}{40t^2\xi^5}|\widehat{f}(t,\xi)|^4\widehat{f}(t,\xi)+\frac{c_1\mathrm{i}}{t^2\xi^5}e^{-\frac{624\mathrm{i}t\xi^5}{625}}\widehat{f}(t,\xi/5)^5\\
	&\quad+\frac{c_2\mathrm{i}}{t^2\xi^5}e^{-\frac{80\mathrm{i}t\xi^5}{81}}|\widehat{f}(t,\xi/3)|^2\widehat{f}(t,\xi/3)^3+\tilde{R}(t,\xi),
	\end{aligned}
	\end{equation*}
where \(\tilde{R}(t,\xi)\) may include \(R_0(t,\xi), R_1(t,\xi)\) and \(R_2(t,\xi)\), but not \(R_3(t,\xi)\).

Let \(t\in [t_1,t_2]\).	Similar to \eqref{50}, we have  	
\begin{equation*}
\begin{aligned}
\partial_t\tilde{w}(t,\xi)
&=e^{\mathrm{i}B(t,\xi)}\bigg(\frac{c_1\mathrm{i}}{t^2\xi^5}e^{-\frac{624\mathrm{i}t\xi^5}{625}}\widehat{f}(t,\xi/5)^5
+\frac{c_2\mathrm{i}}{t^2\xi^5}e^{-\frac{80\mathrm{i}t\xi^5}{81}}|\widehat{f}(t,\xi/3)|^2\widehat{f}(t,\xi/3)^3\\
&\quad\quad\quad\quad\quad+\tilde{R}(t,\xi)\bigg).
\end{aligned}
\end{equation*}	
To show \eqref{62.5},  it is enough to estimate
		\begin{equation}\label{63}
		\begin{aligned}
		&\left|\xi^{-5}\int_{t_1}^{t_2}e^{\mathrm{i}B(s,\xi)}e^{-\frac{624\mathrm{i}s\xi^5}{625}}\widehat{f}(s,\xi/5)^5s^{-2}\,\diff s\right|\\
		&+\left|\xi^{-5}\int_{t_1}^{t_2}e^{\mathrm{i}B(s,\xi)}e^{-\frac{80\mathrm{i}s\xi^5}{81}}|\widehat{f}(t,\xi/3)|^2\widehat{f}(t,\xi/3)^3s^{-2}\,\diff s\right|\lesssim \varepsilon_0(2^jt_1^{1/5})^{-(\frac{1}{2}-\delta)},
		\end{aligned}
		\end{equation}
and 
		\begin{equation}\label{64}
		\begin{aligned}
		\int_{t_1}^{t_2}|\tilde{R}(s,\xi)|\,\diff s \lesssim \varepsilon_0(2^jt_1^{1/5})^{-(\frac{1}{2}-\delta)}.
		\end{aligned}
		\end{equation}
	
For \eqref{63}, due to similarity, we only show the second term can be controlled by the desired bound. 
We integrate by parts in \(s\) to deduce
		\begin{equation*}
		\begin{aligned}
		\left|\xi^{-5}\int_{t_1}^{t_2}e^{\mathrm{i}B(s,\xi)}e^{-\frac{80\mathrm{i}s\xi^5}{81}}|\widehat{f}(t,\xi/3)|^2\widehat{f}(t,\xi/3)^3s^{-2}\,\diff s\right|\lesssim \sum_{n=1}^4H_n,
		\end{aligned}
		\end{equation*}
where \(H_n\) under summation are given by 
		\begin{equation*}
		\begin{aligned}
		&H_1=|\xi|^{-10}|\widehat{f}(s,\xi/3)|^5s^{-2}\big|_{s=t_1}^{s=t_2},\\
		&H_2=|\xi|^{-10}\int_{t_1}^{t_2}|\partial_s\widehat{f}(s,\xi/3)||\widehat{f}(s,\xi/3)|^4s^{-2}\,\diff s,\\
		&H_3=|\xi|^{-10}\int_{t_1}^{t_2}|\partial_sB(s,\xi)||\widehat{f}(s,\xi/3)|^5s^{-2}\,\diff s,\\
		&H_4=|\xi|^{-10}\int_{t_1}^{t_2}|\widehat{f}(s,\xi/3)|^5s^{-3}\,\diff s.
		\end{aligned}
		\end{equation*}
Using the bound on \(\|\widehat{f}\|_{L^\infty}\) in \eqref{52}, we can estimate
\begin{equation*}
\begin{aligned}
H_1\lesssim \varepsilon_0^5(2^jt_1^{1/5})^{-10},
\end{aligned}
\end{equation*}
and
		\begin{equation*}
		\begin{aligned}
		H_2&\lesssim|\xi|^{-10}\int_{t_1}^{t_2}\left(\varepsilon_0^5|\xi|^{-5}s^{-2}+R(s,\xi)\right)\varepsilon_0^4s^{-2}\,\diff s\\
		&\lesssim\varepsilon_0^42^{-10j}t_1^{-2}\left(\varepsilon_0^5+\int_{t_1}^{t_2}R(s,\xi)\,\diff s\right)\lesssim \varepsilon_0^9(2^jt_1^{1/5})^{-10}.
		\end{aligned}
		\end{equation*}
Similarly, one can also obtain the bound \(\varepsilon_0^9(2^jt_1^{1/5})^{-15}\) and \(\varepsilon_0^5(2^jt_1^{1/5})^{-10}\) for \(H_3\) and \(H_4\), respectively. 
The above bounds are much stronger than the desired ones since \(2^{j}t_1^{1/5}\gg 1\) and \(\varepsilon_0\in(0,1)\).

		We turn to prove \eqref{64}. Here it should be clear that we now use the bounds in \eqref{52} instead of \eqref{8}, so the small coefficient \(\epsilon_1\) of the estimates  in Lemma \ref{le:3} and Lemma \ref{le:5} should be replaced by \(\epsilon_0\). Go back to \eqref{37}, we then estimate
		\begin{equation*}
		\begin{aligned}
		\int_{t_1}^{t_2} |R_0|\,\diff s
		\lesssim \epsilon_0^52^{-\frac{15}{4}j}\int_{t_1}^{t_2} s^{-\frac{7}{4}}\,\diff s\lesssim \epsilon_0^5(2^jt_1^{1/5})^{-\frac{15}{4}}.
		\end{aligned}
		\end{equation*}
Recalling \eqref{40}, we have 
		\begin{equation*}
		\begin{aligned}
		\int_{t_1}^{t_2}|R_1|\,\diff s\lesssim \epsilon_0^52^j\int_{t_1}^{t_2}  s^{-\frac{7}{4}}\max\left(2^j,s^{-1/5}\right)^{-\frac{19}{4}}\,\diff s
		\lesssim \epsilon_0^5(2^jt_1^{1/5})^{-\frac{15}{4}},
		\end{aligned}
		\end{equation*}
where we have used the fact \(2^{j}t_1^{1/5}\gg 1\) in the last inequality.	Theses two estimates are more sufficient for the desired bound in \eqref{64}. It remains to bound \(R_2\), we only consider \(R_{21}\) since other cases can be analyzed like {\emph{Step 3}} in the proof of Proposition \ref{pr:3}. In view of \eqref{47.5}, using \(2^{j}t_1^{1/5}\gg 1\) again, we have 	
	\begin{equation*}
	\begin{aligned}
	\int_{t_1}^{t_2} |R_{21}|\,\diff s
	&\lesssim \epsilon_0^52^j \int_{t_1}^{t_2} s^{-11/10}2^{-3j/2}\,\diff s
	+ \epsilon_0^52^j \int_{t_1}^{t_2} s^{-\frac{11}{10}+\frac{\delta}{5}}2^{(-\frac{3}{2}+\delta)j}\,\diff s\\
	&\quad+ \epsilon_0^52^j \int_{t_1}^{t_2} s^{-6/5}2^{-2j}\,\diff s
	\lesssim \epsilon_0^5(2^jt_1^{1/5})^{-(\frac{1}{2}-\delta)}.
	\end{aligned}
	\end{equation*}

	We now come to the proof of \eqref{62}. 	Since \(B\) is real, it follows from \eqref{61} that  
		\begin{equation}\label{65}
		\begin{aligned}
		\big||\widehat{f}(t,\xi)|-|w_\infty(\xi)|\big|\leq \varepsilon_0(|\xi|t^{1/5})^{-(\frac{1}{2}-\delta)}.
		\end{aligned}
		\end{equation}
Let 
		\begin{equation}\label{66}
		\begin{aligned}
		A(t,\xi):=B(t,\xi)+\frac{1}{40t\xi^5}|\widehat{f}(t,\xi)|^4.
		\end{aligned}
		\end{equation}
A direct calculation shows that
		\begin{equation*}
		\begin{aligned}
		A(t_2,\xi)-A(t_1,\xi)&=\frac{1}{40\xi^5} \int_{t_1}^{t_2}\left(|\widehat{f}(s,\xi)|^4-|\widehat{f}(t_2,\xi)|^4\right)\frac{\diff s}{s^2}\\
		&\quad-\frac{1}{40t_1\xi^5}\left(|\widehat{f}(t_1,\xi)|^4-|\widehat{f}(t_2,\xi)|^4\right).
		\end{aligned}
		\end{equation*}
This together with \eqref{65} implies that \(A(t,\xi)\) is a Cauchy sequence in time. Therefore there exists \(A_\infty\in L^\infty\) such that 
		\begin{equation}\label{67}
		\begin{aligned}
		|A(t,\xi)-A_\infty(\xi)|\lesssim \varepsilon_0(|\xi|t^{1/5})^{-(\frac{1}{2}-\delta)}.
		\end{aligned}
		\end{equation}		
Combining \eqref{65}, \eqref{66} and \eqref{67}, we obtain
		\begin{equation*}
		\begin{aligned}
		\bigg|B(t,\xi)-\bigg(A_\infty(\xi)-\frac{1}{40t\xi^5}|w_\infty(\xi)|^4\bigg)\bigg|\lesssim \varepsilon_0(|\xi|t^{1/5})^{-(\frac{1}{2}-\delta)}.
		\end{aligned}
		\end{equation*}
So this resulting estimate together with \eqref{48}-\eqref{49} yields 
		\begin{equation*}
		\begin{aligned}
		\bigg|\widehat{f}(t,\xi)-w_\infty(\xi)\exp\bigg(-\mathrm{i}A_\infty(\xi)+\frac{\mathrm{i}}{40t\xi^5}|w_\infty(\xi)|^4\bigg)\bigg|\lesssim \varepsilon_0(|\xi|t^{1/5})^{-(\frac{1}{2}-\delta)}.
		\end{aligned}
		\end{equation*}
We finally define \(f_\infty(\xi):=w_\infty(\xi)\exp\big(-\mathrm{i}A_\infty(\xi)\big)\) to conclude \eqref{62}. 
		
	\end{proof}

\section*{Acknowledgments}
 The author is grateful to Jean-Claude Saut for many helpful suggestions and acknowledges the support of the ANR project ANuI. The author would also like to thank Benjamin Harrop-Griffiths and Mamoru Okamoto for kindly sharing their expertises on the testing by wave packets in their works.

\end{document}